\documentclass[12pt,reqno]{amsart}
\usepackage[top=1in,bottom=1in,left=1.1in,right=1.1in]{geometry}
\usepackage{color}

\usepackage{amssymb}
\usepackage{amsmath}
\usepackage{amsfonts}
\usepackage{geometry}
\usepackage{graphicx}
\usepackage{mathrsfs,amssymb}
\usepackage{stmaryrd}
\usepackage{mathtools}

\usepackage{hyperref}
\usepackage{cleveref}

\usepackage{enumerate}
\usepackage{enumitem}

 \newtheorem{theorem}{Theorem}[section]
 \newtheorem{proposition}[theorem]{Proposition}
 
 \newtheorem{lemma}[theorem]{Lemma}
 \newtheorem{corollary}[theorem]{Corollary}

 \theoremstyle{definition}

 \theoremstyle{remark}
 \newtheorem{remark}[theorem]{Remark}

\numberwithin{equation}{section}

\renewcommand{\Re}{\operatorname{Re}}

\newcommand{\rg}{\operatorname{rg}}

\newcommand{\R}{\mathbb{R}}
\newcommand{\la}{\lambda}

\newcommand{\mb}[1]{\textbf{#1}}
\newcommand{\bm}[1]{{\bf #1}}

\newcommand{\mc}[1]{\mathcal{#1}}

\newcommand{\B}{\mathbb{B}}
\newcommand{\inner}[1]{\langle #1 \rangle}
\newcommand{\norm}[1]{\| #1 \|}
\newcommand{\hsob}[1]{\dot{H}^{#1}(\R^n)}
\newcommand{\core}{C^{\infty}_{c,\text{rad}}(\R^n)}

\title[]{Globally stable blowup profile for supercritical wave maps in all dimensions}

\author{Irfan Glogi\'c}
\address{Department of Mathematics, University of Vienna, Oskar-Morgenstern-Platz 1, 1090 Vienna, Austria}
\email{irfan.glogic@univie.ac.at}

\thanks{The author acknowledges support by the Austrian Science Fund FWF, Projects P 30076 and P 34378.}

\usepackage{kantlipsum}

\begin{document}

\begin{abstract}
	
We consider wave maps from the $(1+d)$-dimensional Minkowski space into the $d$-sphere. It is known from the work of Bizo\'n and Biernat \cite{BizBie15} that in the energy-supercritical case, i.e., for $d \geq 3$, this model admits  a closed-form corotational self-similar blowup solution. We show that this blowup profile is globally nonlinearly stable for all $d \geq 3$, thereby verifying a perturbative version of the conjecture posed in \cite{BizBie15} about the generic large data blowup behavior for this model. To accomplish this, we develop a novel stability analysis approach based on similarity variables posed on the whole space $\R^d$. As a result, we draw a general road map for studying spatially global stability of self-similar blowup profiles for nonlinear wave equations in the radial case for arbitrary dimension $d  \geq 3$.
\end{abstract}

\maketitle
\section{Introduction}
\noindent 
 We consider maps $U$ from the Minkowski space $(\R^{1+d},\eta)$ into the $d$-sphere $(\mathbb{S}^d,g)$.
Here, $(\eta_{\alpha\beta})=\text{diag} (-1,1,\dots,1)$ in the Euclidean coordinate system on $\R^{1+d}$, and $g$ is the standard round metric on $\mathbb{S}^d$.
 We say that $U$ is a \emph{wave map} if it is a critical point (under compactly supported variations) of the  functional
\begin{equation}\label{Eq:Functional}
	S[U]:=\frac{1}{2} \int_{\R^{1+d}} \text{tr}_{\eta}(U^*g) \, \text{dvol}_\eta,
\end{equation}
where $\text{tr}_{\eta}(U^*g)$ is the trace (with respect to $\eta$) of the metric $g$ when pulled back to $\R^{1+d}$ via $U$. To get a more explicit formulation, we put local coordinates on the domain and the target. First, we define the so-called \emph{normal coordinates} $U=(U^1,\dots,U^d)$ on $\mathbb{S}^d$ by letting
\begin{equation}\label{Def:NormCoord}
	U^i:= \theta \hspace{-0.5mm} \cdot \hspace{-0.5mm} \Omega^i, \quad \text{for} \quad i=1,\dots,d,
\end{equation}
where $\theta$ is the polar angle, and $\Omega^i$ are the coordinate functions of the embedding $\mathbb{S}^{d-1} \hookrightarrow \R^{d}$. Then, by placing the standard Euclidean coordinates on the Minkowski space, we get that the Euler-Lagrange equations associated to \eqref{Eq:Functional} take the following form\footnote{Here, and throughout the paper,  the Einstein summation convention is in force, with Greek indices running from $0$ to $d$, and Latin indices going from $1$ to $d$. Also, the indices are raised and lowered with respect to the Minkowski metric $\eta$.}
\begin{equation}\label{Eq:Euler-Lagrange}
	\partial^\mu \partial_\mu U^i	+ \Gamma^i_{jk}(U)\partial^\mu U^j \partial_\mu U^k=0, \quad i=1 \dots d,
\end{equation}
where $\Gamma^i_{jk}$ are the Christoffel symbols for the metric $g$ in normal coordinates \eqref{Def:NormCoord}. We note that \eqref{Eq:Euler-Lagrange} is a coupled nonlinear hyperbolic system of partial differential equations, and the main mathematical interest lies in the underlying Cauchy problem. Namely, one desires to study the following initial value problem for the vector function $U=U(t,x)$
\begin{equation}\label{Eq:WM}
	\begin{cases}
		~\partial^\mu \partial_\mu U^i	+ \Gamma^i_{jk}(U)\partial^\mu U^j \partial_\mu U^k=0, \quad i=1 \dots d, \\
		~U(0,\cdot)=U_0,\\
		~\partial_t U(0,\cdot) = U_1,
	\end{cases}
\end{equation}
where $(t,x) \in I\times \R^d$ for some interval $I \subseteq \R$ containing zero, and $U_0,U_1: \R^d \rightarrow \R^d $ are the initial data.

In view of the nonlinear nature of the problem \eqref{Eq:WM}, the central question is the one of regularity versus blowup: Can smooth and localized initial data lead to singularity formation in finite time? A positive answer to this question then gives rise to an important follow-up problem, namely the one of determining all of the stable blowup profiles. 
In regard to the existence question, there is a useful heuristic principle, which is based on the relation between the nonlinear scaling and the conserved quantities. To expand on this, we first note that the system \eqref{Eq:WM} is invariant under the rescaling $U \mapsto U_\la$, where
\begin{equation}\label{Def:Non_scaling}
	U_\la(t,x):= U( t/\la, x/\la ), \quad \la >0.
\end{equation}
Furthermore, we have that the \emph{energy}
\begin{equation*}
	E[U](t):=\frac{1}{2} \int_{\R^{d}} \delta^{\alpha \beta} g_{ij}(U(t,x)) \partial_\alpha U^i(t,x) \partial_{\beta} U^j(t,x) dx
\end{equation*}
is conserved under the flow, and furthermore scales under  \eqref{Def:Non_scaling} in the following way
\begin{equation*}
	E[U_\la](t)= \la^{d-2} E[U](t/\la).
\end{equation*}
Consequently, the wave maps system \eqref{Eq:WM} is \emph{energy-subcritical} when $d=1$, \emph{energy-critical} when $d=2$, and \emph{energy-supercritical} when $d \geq 3$. 
The aforementioned heuristic principle says that the subcritical equations are globally regular, while the supercritical ones do admit blowup. 
That this heuristic is true for \eqref{Eq:WM} when $d=1$ has been known since 1980's, from the works of Gu \cite{Gu80}, Ginibre and Velo \cite{GinVel82}, and Shatah \cite{Sha88}. In contrast to this, in both the critical and the supercritical case blowup is possible. This can be observed already for a special class of the so-called \emph{corotational} maps. We say that a wave map $U$ is corotational if
\begin{equation}\label{Def:CorAnsatz}
	U^i(t,x) = u(t,|x|) \, x^i , \quad i = 1 \dots d,
\end{equation}
for some $u(t,\cdot):[0,\infty) \rightarrow \R$, which we call the \emph{radial profile} of $U(t,\cdot)$. Under this ansatz, \eqref{Eq:WM} reduces to a Cauchy problem for a single $(d+2)$-dimensional radial semilinear wave equation for the profile $u=u(t,r)$
\begin{equation}\label{Eq:CorWM}
	\begin{cases}
		~\displaystyle{\partial^2_{t}u-\partial_r^2u-\frac{d+1}{r}\partial_ru + \frac{d-1}{2r^3}\big( \hspace{-0.5mm}\sin(2ru)-2ru\big)=0},\\ \smallskip
		~u(0,\cdot)=u_0, \\
		~\partial_t u(0,\cdot)=u_1,
	\end{cases}
\end{equation}
where the initial condition is obtained from the corotational initial data
\[
U_0(x)= u_0(|x|)  x, \quad U_1(x)= u_1(|x|)  x.
\]
In the critical case, $d=2$, the problem \eqref{Eq:CorWM} has been studied by Bizo\'n, Chmaj and Tabor \cite{BizChmTab01}, who numerically observed finite time blowup at the origin for large initial data. This was then followed by rigorous constructions by Krieger, Schlag, and Tataru \cite{KriSchTat08}, and Rapha\"el and Rodnianski \cite{RapRod12}.

Exhibiting blowup in the supercritical case, $d \geq 3$, is easier, as there exist self-similar solutions. The first construction is due to Shatah \cite{Sha88}, who proved existence for $d=3$ by variational techniques. Two years later, what is believed to be the Shatah's solution was found in closed form by Turok and Spergel \cite{TurSpe90}. 
%
Then, relatively recently, 
  Bizo\'n and Biernat \cite{BizBie15}
 discovered the analogue of the Turok-Spergel profile in higher dimensions. Namely, they noted that \eqref{Eq:CorWM} admits for all $d\geq 3$ the following solution
\begin{equation}\label{Def:BB_sol}
u_T(t,r):=\frac{1}{T-t}\phi\left(\frac{r}{T-t}\right), \quad \phi(\rho)= \frac{2}{\rho}\arctan\left(\frac{\rho}{\sqrt{d-2}}\right), \quad T>0.
\end{equation}
Note that this, via \eqref{Def:CorAnsatz}, yields for the system \eqref{Eq:WM} a self-similar corotational solution 
\begin{equation}\label{Def:BB_sol_vector}
U_T(t,x):=  \Phi\left(\frac{x}{T-t}\right), \quad \Phi(x)= \phi(|x|)  x, \quad T>0,
\end{equation}
which suffers gradient blowup at the origin as $t \rightarrow T^-$.

\subsection{The main result}
The role of the solution \eqref{Def:BB_sol_vector} for generic Cauchy evolutions of \eqref{Eq:WM} was first studied from the numerical point of view by Bizo\'n, Chmaj and Tabor \cite{BizChmTab00} for $d=3$, and then by Bizo\'n and Biernat \cite{BizBie15} for higher dimensions. 
In both cases, the authors made the observation that generic large corotational data evolutions lead to blowup via \eqref{Def:BB_sol_vector}. In this paper we rigorously prove the perturbative version of these observations, i.e., we establish the global-in-space nonlinear stability of the profile $U_T$ for all $d \geq 3$. More precisely, we show  the existence of an open set (in a suitably chosen topology given by Sobolev norms) of corotational initial data around 
\begin{equation}\label{Eq:Cor_initial_data}
	U_1(0,\cdot)= \Phi, \quad \partial_tU_1(0,\cdot) = \Lambda \Phi, \quad \text{where} \quad  \Lambda f(x): = x  \cdot \nabla f(x),
\end{equation}
for which the Cauchy evolution of \eqref{Eq:WM} forms a singularity in finite time $T$ by converging to $U_T$ (i.e., to $\Phi$, after undoing the self-similar scaling) globally in space in the underlying norm as well as uniformly on compact sets. The formal statement is as follows.

\begin{theorem}\label{Thm:Main}
	Let $d \geq 3$ and $s,k >0$ such that
	\begin{equation}\label{Eq:s,k}
		\frac{d}{2} < s < \frac{d}{2} + \frac{1}{2d}, \quad k > d+2, \quad k \in \mathbb{N}.
	\end{equation}
	Then there exists $\varepsilon>0$ such that for any corotational initial data of the form
	\begin{equation}\label{Eq:Init_data_main_result}
		(U_0,U_1) = (\Phi, \Lambda \Phi) + (\varphi_0,\varphi_1),
	\end{equation}
	where $\varphi_0,\varphi_1$ are Schwartz functions for which
	\begin{equation}\label{Eq:Init_data_small}
		\| (\varphi_0,\varphi_1) \|_{\dot{H}^{s} \cap \dot{H}^{k}(\R^d) \, \times \, \dot{H}^{s-1} \cap \dot{H}^{k-1} (\R^d)} < \varepsilon,
	\end{equation}
	there exists $T>0$ and a unique classical solution $U \in C^\infty([0,T)\times \R^d)$ to \eqref{Eq:WM}, which blows up at the origin as $t \rightarrow T^-$. Furthermore, the following profile decomposition holds
	\begin{equation}\label{Eq:Profile_decomposition}
		U(t,x)=\Phi\left(\frac{x}{T-t}\right) + \varphi\left(t,\frac{x}{T-t}\right),
	\end{equation}
where
\begin{equation}\label{Eq:varphi_zero}
	\|  \varphi(t,\cdot) \|_{\dot{H}^{r}(\R^d)} + \norm{((T-t)\partial_t +\Lambda) \varphi(t,\cdot)}_{\dot{H}^{r-1}(\R^d)}  \rightarrow 0 \quad \text{as} \quad  t \rightarrow T^-
\end{equation}
for all $r \in [s,k]$. In particular,
\begin{equation}\label{Eq:uniform_on_compact}
	  U(t,(T-t)\cdot) \rightarrow \Phi
\end{equation}
locally uniformly in $\R^d$ as $t \rightarrow T^-$.
\end{theorem}

 \begin{remark}
 	The seemingly strange condition on $s,k$ serves a twofold purpose. First, imposing $s,k > \frac{d}{2}$ ensures the decay of the linear flow in the topology given by the norm \eqref{Eq:Init_data_small}. Second, the additional restrictions in \eqref{Eq:s,k} guarantee that the underlying nonlinear operator is locally Lipschitz continuous in this topology; see Lemma \ref{Lem:Nonlin_est} and Proposition \ref{Prop:Schauder}.
 \end{remark}

\begin{remark}
	The profile decomposition \eqref{Eq:Profile_decomposition}, together with \eqref{Eq:varphi_zero}, shows that the evolution of the perturbation \eqref{Eq:Init_data_main_result}, when dynamically self-similarly rescaled, converges back to $(\Phi,\Lambda \Phi)$ in $\dot{H}^r \times \dot{H}^{r-1}(\R^d)$. This is precisely what we mean by the stability of the self-similar blowup via $U_T$.
\end{remark}

\begin{remark}
	We note that \eqref{Eq:uniform_on_compact} corresponds precisely to what has been observed numerically in \cite{BizBie15}; see Figure 1 there. What is more,  \eqref{Eq:uniform_on_compact} is a fundamentally new rigorous insight, which stems from our global-in-space framework, and does not follow from the existing stability analysis approaches; see more on previous work in Section \ref{Sec:Related_results} below. This follows from the local $L^\infty$-embedding of the intersection Sobolev space in \eqref{Eq:Init_data_small}, although it also straightforwardly follows from our stability analysis, as a consequence of convergence of the rescaled radial profile $u(t,(T-t)\cdot)$ to $\phi$ globally uniformly on $[0,\infty)$. 
\end{remark}
 
 \subsection{Related results and discussion}\label{Sec:Related_results}
 
 To put our work in perspective, we first briefly outline the most important developments on the problem of stability of self-similar blowup for wave maps. Then, we illustrate how the approach of this paper complements and improves upon the existing stability analysis techniques.
 
 Stability of self-similar solutions for wave maps was first explored by Bizo\'n. In \cite{Biz00} he developed heuristics for the blowup stability analysis in lightcones (see also \cite{BizChmTab00}), and in \cite{Biz05} he formulated and numerically studied the spectral problem that underlies the linear stability analysis of the solution \eqref{Def:BB_sol_vector}. Putting Bizo\'n's heuristics on rigorous ground was later taken up by Donninger \cite{Don10a,Don09,Don10}, whose approach turned out quite fruitful, resulting in a multitude of results on self-similar blowup stability in lightcones; this corresponds to showing convergence of type \eqref{Eq:uniform_on_compact}, but on the unit ball $\B^d$. Here, we survey the results that pertain to wave maps. The first nonlinear stability result for \eqref{Def:BB_sol_vector} was established for $d=3$ in a pair of papers by Donninger \cite{Don11}, and by Aichelburg, Donninger and Sch\"orkhuber \cite{DonSchAic12}. These results, however, assume spectral stability of \eqref{Def:BB_sol_vector}, which hinges on a difficult non-self-adjoint spectral problem. A resolution of this problem appeared several years later in a paper by Costin, Donninger and Xia \cite{CosDonXia16}, and then in a significantly simpler form in the work of Costin, Donninger and the author \cite{CosDonGlo17}. What is more, \cite{CosDonGlo17} establishes the spectral stability of \eqref{Def:BB_sol_vector} for all $d \geq 3$, and furthermore sets up a general scheme for showing similar spectral results in wider semilinear contexts. For additional descriptions of this method, see the author's thesis \cite{Glo18}, or his work with Sch\"orkhuber \cite{GloSch21}. Later on, using the results from \cite{CosDonGlo17}, Chatzikaleas, Donninger and the author \cite{ChaDonGlo17} established the nonlinear stability of \eqref{Def:BB_sol_vector} in all higher odd dimensions. We point out that, due to the local nature of the underlying coordinate frame, \cite{ChaDonGlo17} gives stability of \eqref{Def:BB_sol_vector} along the shrinking time slices $\{t\} \times \B^d_{T-t}$.  How to extend \cite{ChaDonGlo17} to even dimensions is not clear at the moment, due to the apparent lack of befitting inner products that characterize Sobolev norms on bounded domains, and which are necessary in order to establish decay of the linearized wave maps flow in lightcones. 
  Another, more recent stability analysis approach was put forward by Biernat, Donninger and Sch\"orkhuber  \cite{BieDonSch21}. They use the so-called hyperboloidal similarity variables to keep track of the evolution up to the future lightcone of the blowup point, and they thereby prove for $d=3$ stability of \eqref{Def:BB_sol_vector} in regions that include portions of the spacetime that go beyond the time of blowup.
 For an extension of this work to higher odd dimensions, see Donninger and Ostermann \cite{DonOst21}. For a recent result on stability of \eqref{Def:BB_sol_vector} at the optimal regularity, see Donninger and Wallauch \cite{DonWal22}.
 
 
 In this paper we undertake a new line of attack to the problem of stability of self-similar blowup for nonlinear wave equations. The key novel feature is a global coordinate frame given by the similarity variables that are posed on the whole space $\R^d$. This approach therefore allows for capturing the evolution of perturbations of self-similar profiles globally on $\R^d$. Consequently, in contrast to the existing approaches, we get stability along horizontal and spatially global time slices $\{ t \} \times \R^d$. In particular, as a simple by-product, we get that blowup can not happen outside the origin. Furthermore, the fact that we are working on $\R^d$ allows us to utilize the Fourier analysis techniques and thereby develop a unified framework for all dimensions $d \geq 3$. Additionally, in terms of further applications, we point out that our approach provides the necessary framework for studying the self-similar threshold for blowup phenomena (see, e.g., \cite{Biz01,BieBizMal17,GloMalSch20}), where it is required to consider perturbations on the whole space $\R^d$.  
As a result, in addition to proving Theorem \ref{Thm:Main}, in this paper we draw a general road map for studying global radial stability of self-similar blowup profiles for nonlinear wave equations in arbitrary dimension $d  \geq 3$. Before we outline the proof of the main result, and thereby chart out our approach, we include a section on the historical developments in the study of wave maps.

 \subsection{Some history on wave maps}

In the most general setting, wave maps are defined on a Lorentzian manifold with values in a Riemannian manifold. They therefore represent the hyperbolic analogue of the much studied harmonic maps. 
In the physics literature, wave maps appeared as models for particle physics, and go under the name of (nonlinear) $\sigma$-models.
 For other physical applications, and for more historical remarks, see Misner \cite{Mis78,Mis82}.
From the mathematical point of view, wave maps have attracted a great deal of attention over the past several decades. 
The type that has garnered most attention so far are the ones with the flat Minkowski space domain,\footnote{The theory for the Minkowski space is already so complicated that so far little attention has been devoted to the general case. What is more, some authors even define wave maps as the ones on the flat Minkowski space.}
and the basic mathematical questions for the underlying Cauchy problem concern local well-posedness, global regularity and finite time blowup. 

 \emph{Local well-posedness}:
 Since the wave maps system is semilinear, local well-posedness in Sobolev spaces of high order is standard; see Ginibre and Velo \cite{GinVel82} for a result in  $H^s \times H^{s-1}(\R^d)$ with integer $s \geq  \lfloor\frac{d}{2}\rfloor + 2$.  The first well-posedness result at the critical regularity $s=\frac{d}{2}$ under symmetry restrictions is due to Shatah and Tahvildar-Zadeh \cite{ShaTah94}. Later on, assuming no symmetries, Klainerman and Machedon \cite{KlaMac97a}, Klainerman and Selberg \cite{KlaSel97}, and Keel and Tao \cite{KeeTao98} established well-posedness of ~\eqref{Eq:WM} (for more general targets) for $s > \frac{d}{2}$.
Subsequently, Tataru \cite{Tat98,Tat01} proved well-posedness outside symmetry for data in the scaling critical Besov space for all $d\geq2$. In line with this, see also a recent paper by Candy and Herr \cite{CanHer18}. For an ill-posedness result at the critical regularity for $d=1$ see Tao \cite{Tao00}. For a more detailed overview, and an in-depth discussion of the question of well-posedness, see the survey paper by Tataru \cite{Tat04}.


 \emph{Global regularity}: Although for $d \geq 2$ wave maps do admit blowup, it is expected that smooth data that are small in the scaling invariant Sobolev space $\dot{H}^{{d/2}}\times\dot{H}^{{d/2}-1}$, lead to global regularity, irrespective of the target geometry; see Klainerman  \cite{Kla97}. First results of this type are due to Tao \cite{Tao01,Tao01a} for spherical targets. The extensions to more general manifolds are due to Klainerman and Rodnianski \cite{KlaRod01}, Krieger \cite{Kri03a}, Nahmod, Stefanov and Uhlenbeck \cite{NahSteUhl03} and Shatah and Struwe \cite{ShaStr02}. For large data evolutions, the main problem is to determine the types of manifolds that ensure global regularity. According to a heuristic principle, this is expected to hold for negatively curved targets, in view of their defocusing effect. Confirmation of this heuristic in equivariance symmetry is due to Shatah and Tahvildar-Zadeh \cite{ShaTah92} for $d = 2$. Without symmetry restrictions, however, matters are more complicated, and most of the effort so far has been employed for the target of constant negative curvature, the hyperbolic space. As a result, the global regularity for $d=2$ was settled in a series of works by Tao \cite{Tao09}, Sterbenz and
Tataru \cite{SteTat10,SteTat10a} and Krieger and Schlag \cite{KriSch12}. For supercritical dimensions, $d \geq 3$, it turns out that mere negative curvature is not enough to preclude blowup, even in the case of corotational maps. This was shown by Cazenave, Shatah and Tahvildar-Zadeh \cite{ShaStr98}  for $d = 7$, and then  for $d \geq 8$ by Donninger and the author \cite{DonGlo19}, who also showed stability of the exhibited blowup. In the complementary case, $3 \leq d \leq 6$, the existence of blowup is still open. 

 \emph{Finite time blowup}: For positively curved targets, in view of the underlying focusing effect, large data evolutions are expected to form singularities. 
 In the critical case, $d=2$, the first rigorous constructions for the sphere are due to Krieger, Schlag and Tataru \cite{KriSchTat08}, and Rodnianski and Sternbenz \cite{RodSte10}. The former result was extended by Gao and Krieger \cite{GaoKri15}, while stability of their solutions was shown by Krieger and Miao \cite{KriMia20}. Another notable construction is due to Raph\"el and Rodnianski \cite{RapRod12}, who also prove stability of the constructed blowup profile. In the supercritical case, $d \geq 3$, exhibiting blowup is easier, as self-similar solutions exist. The earliest result goes back to Shatah \cite{Sha88} in $d=3$ for the sphere. Later on, Bizo\'n \cite{Biz00} proved that in addition to Shatah's profile there are infinitely many corotational self-similar solutions into the sphere. Then, together with Biernat and Maliborski \cite{BieBizMal17}, he extended this result to $4 \leq d \leq 6$, and furthermore conjectured that for $d \geq 7$ there exists exactly one corotational self-similar profile, namely \eqref{Def:BB_sol_vector}. Interestingly, for the same range, $d \geq 7$, a new, non-self-similar blowup mechanism was discovered by Ghoul, Ibrahim and Nguyen \cite{GhoIbrNgu18}.

 \subsection{Outline of the proof of the main result}
 
 Since the system \eqref{Eq:WM} in corotational symmetry reduces to the radial wave equation \eqref{Eq:CorWM}, the bulk of our work consists of showing stability of the solution \eqref{Def:BB_sol}. Then, by using the equivalence of norms of corotational maps and their radial profiles, we turn the obtained stability result into Theorem \ref{Thm:Main}.
 %
In Section \ref{Sec:Simil_var}, we first, for convenience, let $v(t,\cdot):=u(t,|\cdot|)$ and thereby reformulate \eqref{Eq:CorWM} into a semilinear wave equation for $v$ in $n:=d+2$ dimensions
 \begin{equation}\label{Eq:NLW_intro}
 	\displaystyle{	~~\partial^2_t v - \Delta v = \frac{n-3}{2|x|^3}}\left(2|x|v - \sin (2|x|v)\right),
 \end{equation}
  with radial initial data
  \begin{equation}\label{Eq:Init_data_intro}
  	v_0=u_0(|\cdot|) \quad \text{and} \quad v_1=u_1(|\cdot|).
  \end{equation}
   Then, we introduce the (global) similarity variables 
 \begin{equation*}\label{Def:SS_var_intro}
 	\tau=\tau(t):=\ln \left(\frac{T}{T-t}\right), \quad \xi=\xi(t,x):=\frac{x}{T-t}.
 \end{equation*}
 Note that, in this way the time slab $[0,T) \times \R^n$ is mapped via $(t,x) \mapsto (\tau,\xi)$ into the half-space $[0,\infty) \times \R^n$.
 Furthermore, by also rescaling the dependent variables
 \begin{equation*}\label{Def:scaled_v_intro}
 	(T-t)v(t,x)=:\psi_1(\tau,\xi) \quad \text{and} \quad (T-t)^2\partial_tv(t,x)=:\psi_2(\tau,\xi),
 \end{equation*}
 the self-similar solution
 \begin{equation}\label{Def:SS_sol_intro}
 	\big(v_T(t,\cdot),\partial_t v_T(t,\cdot)\big):=\big(u_T(t,|\cdot|),\partial_t u_T(t,|\cdot|)\big)
 \end{equation} becomes $\tau$-independent; we denote it by $\Psi_{0}=\Psi_{0}(\xi)$. 
 Most importantly, in this way we turn the problem of stability of finite time blowup of \eqref{Def:SS_sol_intro}, into a generally more familiar problem of the asymptotic stability of the static solution $\Psi_{0}$. However, to carry out the stability analysis of $\Psi_0$, we first need a well-posedness theory for the Cauchy evolution in the new variables. To this end, we first let $\Psi(\tau):=\big(\psi_1(\tau,\cdot),\psi_2(\tau,\cdot) \big)$. By that, the problem \eqref{Eq:NLW_intro}-\eqref{Eq:Init_data_intro} adopts the following first order vector form
 \begin{equation}\label{Eq:WM_system_abstract_intro}
 	\begin{cases}
 		\partial_\tau\Psi(\tau)=\widetilde{\mb L}_0\Psi(\tau)+\mb N_0(\Psi(\tau)),\\	
 		\Psi(0)=\mb U_0(T),
 	\end{cases}
 \end{equation}
where $ \widetilde{\mb L}_0$ is the $n$-dimensional
 wave operator in similarity variables (see \eqref{Def:Wave_oper_sim}), $\mb N_0$ is the remaining nonlinear operator, and the initial data are $\mb U_0(T) = \big( T v_0(T\cdot) , T^2v_1(T\cdot) \big).$
 In order to study dynamics near the static profile $\Psi_{0}$, we consider the perturbation ansatz $\Psi(\tau)=\Psi_{0}+\Phi(\tau)$. This leads to the central Cauchy problem of the paper
 \begin{equation}\label{Eq:AbsEvolPhi}
 	\begin{cases}
 		\partial_\tau \Phi(\tau)=\widetilde{\mb L}_0\Phi(\tau)+\mb{L}'\Phi(\tau)+{\mb N}(\Phi(\tau)),\\
 		\Phi(0)= \mb U (T).
 	\end{cases}
 \end{equation}
 Here $\mb L'$ is the Fr\'echet derivative of $\mb N_0$ at $\Psi_0$, $\mb N$ is the remaining nonlinear operator, and the initial data are $\mb U (T)= \mb U_0(T)- \Psi_0$. 
 To study evolutions of \eqref{Eq:AbsEvolPhi}, we need an appropriate functional setup. To that end, in Section \ref{Sec:Funct_setup} we introduce the principal function space of the paper
 \begin{equation*}
 	\mc H^{s,k}:= \dot{H}_{\text{rad}}^{s} \cap \dot{H}_{\text{rad}}^{k}(\R^n) \, \times \, \dot{H}_{\text{rad}}^{s-1} \cap \dot{H}_{\text{rad}}^{k-1}(\R^n) ,
 \end{equation*}
with exponents $s,k$ to be fixed later.
 To construct solutions to \eqref{Eq:AbsEvolPhi} in $\mc H^{s,k}$, we take up the abstract semigroup approach. 
 In Section \ref{Sec:Free_semigroup} we concentrate on the free operator $\widetilde{\mb L}_0$, and prove the central result of the linear theory. Namely, we show that for every $n \geq 3$, the operator $\widetilde{\mb L}_0$, being initially defined on test functions, is closable, and its closure $\mb L_0$ generates a strongly continuous semigroup $\mb S_0(\tau)$  on $\mc H^{s,k}$; see Proposition \ref{Prop:Free_semigroup}. The proof is  based on the Lumer-Phillips theorem, and it entails a delicate construction of global, radial and decaying solutions to a certain degenerate elliptic equation; see Lemma \ref{Lem:Closed_range}. After establishing this result, it is not difficult to propagate it  perturbatively to $\mb L:=\mb L_0 + \mb L'$. Namely, in Section \ref{Sec:Perturbed_semigroup}, we show that the operator $\mb L': \mc H^{s,k} \rightarrow \mc H^{s,k}$ is compact, and then we readily infer that $\mb L$ too generates a semigroup on $\mc H^{s,k}$. We denote this semigroup by $\mb S(\tau)$, and we use it to recast \eqref{Eq:AbsEvolPhi} in the integral form
 \begin{equation}\label{Eq:Duhamel1}
 	\Phi(\tau)=\mathbf{S}(\tau)\mathbf{U}(T)+\int_{0}^{\tau}\mathbf{S}(\tau-s)\mathbf{N}(\Phi(s))ds.
 \end{equation}
 Note that the asymptotic stability of $\Psi_0$ would follow from the global existence and decay of solutions to \eqref{Eq:Duhamel1} for small data. This, of course, necessitates decay of $\mb  S(\tau)$, and this is where the specific choice of $s,k$ starts to matter. The first condition we impose is $s,k > \frac{d}{2}$, as then the free semigroup $\mb S_0(\tau)$ obeys exponential decay in $\mc H^{s,k}$. Still, perturbing the generator $\mb L_0$ by $\mb L'$ in general induces growth for $\mb S(\tau)$. As it turns out, since $\mb L'$ is compact, the principle of linear stability holds in this setting. Namely, any exponential growth of $\mb S(\tau)$ is necessarily caused by the presence of the unstable point spectrum of $\mb L$. This follows from a general spectral mapping theorem that applies to this setting; see \cite{Glo22}, Theorem B.1.
   Consequently, in Section \ref{Sec:Spec_analy_L} we focus on the spectral analysis of the operator $\mb L$. We emphasize that, due to the highly non-self-adjoint character of $\mb L$, studying its spectrum is an extremely difficult problem. Fortunately, we are able to reduce this analysis to the one done in earlier works on stability in lightcones. In particular, by using the results from \cite{CosDonGlo17}, we show that $\mb L$ has precisely one unstable spectral point, $\la=1$, which is a simple eigenvalue; see Proposition \ref{Prop:Pert_semigroup}.
 This eigenvalue is, however, given rise to by the underlying time translation symmetry, and is hence not a genuine instability. We therefore introduce the Riesz projection $\mb P$ associated to it, and then from \cite{Glo22}, Theorem B.1 we infer the exponential decay of $\mb S(\tau)$ on the stable subspace
 \begin{equation*}\label{Eq:Decay}
 	\|\mb S(\tau)(\mb I  - \mb P) \|_{\mc H^{s,k}} \leq C e^{-\omega \tau}, \quad \omega >0;
 \end{equation*}
see Proposition \ref{Prop:Decay_stable}. With the above decay property at hand, to construct strong solutions to \eqref{Eq:Duhamel1}, it is enough to ensure Lipschitz continuity of $\mb N$ in $\mc H^{s,k}$. This can, in fact, be ensured  with additional restrictions on $s,k$.  In Section \ref{Sec:Est_nonlin}, we show that for all $s$ close to the critical exponent $s_c=\frac{d}{2}$, and for all $k$ integer and large enough, the operator $\mb N$ is locally Lipschitz continuous on $\mc H^{s,k}$. The proof relies on a new Schauder estimate, which we establish in Proposition \ref{Prop:Schauder}. The proof of this estimate, in turn, relies crucially on, to our knowledge, a new inequality for the weighted $L^\infty$-norms of derivatives of radial Sobolev functions; see Proposition \ref{Prop:Strauss}.

With these technical results at hand, in Section \ref{Sec:Strong_sol} we construct for \eqref{Eq:Duhamel1} global, exponentially decaying strong $\mc H^{s,k}$-solutions for small data. To suppress the growth coming from the symmetry-induced eigenvalue $\la=1$, we use a standard Lyapunov-Perron type argument, by means of which we also extract out the blowup time $T$. 
In Section \ref{Sec:Upgrade_to_class}, by regularity arguments we upgrade the constructed strong solutions to the classical ones. By translating the obtained result back to physical coordinates, we get stability of $u_T$. Finally, by means of the equivalence of norms of corotational maps and their radial profiles, from this we derive Theorem \ref{Thm:Main}.
  
 \subsection{Notation and conventions}
 
  By $\mc S(\R^d)$ we denote the  Schwartz space. We also allow for vector functions in this space by requiring that every component be a standard scalar Schwartz function.
 For a Banach space $X$, we denote by $\mc B(X)$ the set of bounded linear operators on $X$. Given a closed linear operator $(L,\mc D(L))$ on $X$,  we denote by $\rho( L)$ the resolvent set of $L$, while $\sigma( L):= \mathbb{C} \setminus \rho( L)$ stands for the spectrum of $ L$, and $\sigma_p( L)$ denotes the point spectrum. Also, for $\la \in \rho( L)$ we use the following convention for the resolvent operator, $R_{ L}(\la):=(\la I -  L)^{-1}$. By $\ker L$ and $\rg L$ we denote respectively the kernel and the range of $L$. Also, we use the usual asymptotic notation $a \lesssim b$ to denote $a \leq Cb$ for some constant $C>0$. Furthermore, we write $a \simeq b$ if $a \lesssim b$ and $b \lesssim a.$ We also use the Japanese bracket notation $\inner{x}= \sqrt{1+|x|^2}$. For the Wronskian of two functions $f,g \in C^1(I), I \subseteq \R$, we use the following convention $W(f,g):=fg'-f'g$.
 
 \section{Equation in similarity variables}\label{Sec:Simil_var}

 \noindent In this section we introduce the similarity variables, which are adapted to the self-similar nature of the solution $u_T$ in \eqref{Def:BB_sol}, and thereby convert the problem of stability of finite time blowup via $u_T$ into the problem of the asymptotic stability of a static profile. For convenience, we introduce some notation. First, we denote $n:=d+2$. Also, we define
 \begin{equation}\label{Def:u}
 	v(t,x):=u(t,|x|),
 \end{equation}
and thereby rewrite the Cauchy problem \eqref{Eq:CorWM} in the following form
 \begin{equation}\label{Eq:NLW}
 	\begin{cases}
 	\displaystyle{	~~\partial^2_t v - \Delta v = \frac{n-3}{2|x|^3}}\left(2|x|v - \sin (2|x|v)\right),\\
 		 \smallskip
 		~~v(0,\cdot)=v_0, \\
 		~~\partial_t v(0,\cdot)=v_1,
 	\end{cases}
 \end{equation}
where $(t,x) \in I \times \R^{n}$, $I \subseteq \R$ is an interval containing zero, and the radial initial data is $v_0=u_0(|\cdot|)$ and $v_1=u_1(|\cdot|)$.
 
\subsection{Similarity variables}
 Given $T>0$, we define
 \begin{equation}\label{Def:Simil_var}
 	\tau= \tau(t):=\ln \left(\frac{T}{T-t} \right),  \quad \xi=\xi(t,x):=\frac{x}{T-t}.
 \end{equation}
Note that transformation \eqref{Def:Simil_var} maps the strip $S_T:=[0,T) \times \R^n$ 
into the upper half-space $H_+:=[0,\infty) \times \R^n$. In addition, we define the rescaled dependent variable
\begin{equation}\label{Def:psi}
	\psi(\tau,\xi):=(T-t)v(t,x)=Te^{-\tau}v(T-Te^{-\tau},Te^{-\tau}\xi).
\end{equation}
Consequently, the evolution of $v$ inside $S_T$ corresponds to the evolution of $\psi$ inside $H_+$. Note that derivative operators with respect to $t$ and $x$ in the new variables become
\begin{equation}\label{Eq:Diff_law}
	\partial_t = \frac{e^\tau}{T}(\partial_{\tau} + \Lambda), \quad \text{and} \quad     \partial_{x_i}= \frac{e^{\tau}}{T}\partial_{\xi_i},
\end{equation}
where the operator $\Lambda$ acts on $\xi$ and is defined in \eqref{Eq:Cor_initial_data}.  Based on \eqref{Eq:Diff_law},
we get that the semilinear wave equation \eqref{Eq:NLW} transforms into
\begin{equation}\label{Eq:NLW_sim_var}
	\big(\partial^2_\tau + 3 \partial_\tau + 2 \Lambda \partial_\tau - \Delta + \Lambda^2 + 3 \Lambda + 2 \big)\psi = \frac{n-3}{2|\xi|^3}\big(2|\xi|\psi - \sin (2|\xi|\psi) \big).
\end{equation}
The linear operator on the left is fairly complicated compared to d'Alembertian in \eqref{Eq:NLW}, and it seems rather hopeless to try and employ the traditional methods of constructing solutions, e.g., the Fourier synthesis. We therefore resort to the abstract approach, via the semigroup theory. To that end, we first write \eqref{Eq:NLW_sim_var} in a vector form. Namely, we define
\begin{equation}\label{Def:psi_1}
	\psi_1(\tau,\xi):=\psi(\tau,\xi), \quad \psi_2(\tau,\xi):=(\partial_\tau + \Lambda + 1)\psi(\tau,\xi),
\end{equation} 
and let $\Psi(\tau):=(\psi_1(\tau,\cdot),\psi_2(\tau,\cdot))$. This yields an evolution equation for $\Psi$
\begin{equation}\label{Eq:Evol_equ}
	\partial_{\tau}\Psi(\tau) = \widetilde{\bm L}_0 \Psi(\tau) + \bm N_0(\Psi(\tau)),
\end{equation}
where
\begin{equation}\label{Def:Wave_oper_sim}
	\widetilde{\bf{L}}_0:= 
	\begin{pmatrix}
	-\Lambda -1 &1\\
	\Delta & -\Lambda - 2
\end{pmatrix}
\end{equation}
is the wave operator in similarity variables, and the nonlinearity is given by
\begin{equation*}
	\mb N_0(\mb u) = 
	\begin{pmatrix}
		0 \vspace{1mm} \\ N_0(\cdot,u_1)
	\end{pmatrix}
	\quad \text{for} \quad  N_0(\xi,u_1)=\frac{n-3}{2|\xi|^3}\eta(|\xi|u_1),
\end{equation*}
where, for convenience, we denoted
\begin{equation}\label{Def:Eta}
	\eta(y):= 2y-\sin(2y).
\end{equation}
Also, the initial data become
\begin{equation*}
	\bm U_0(T):= 
	\begin{pmatrix}
		Tv_0(T\cdot) \\
		T^2 v_1(T\cdot)
	\end{pmatrix}.
\end{equation*}
 Now, since \eqref{Def:BB_sol} solves \eqref{Eq:CorWM}, we have that for all $n \geq 5$ the function
\begin{equation*}
	\Psi_0:=
	\begin{pmatrix}
		\phi_0 \\
		\phi_1
	\end{pmatrix}, \quad \text{where} \quad \phi_0=\phi(|\cdot|) \quad \text{and} \quad \phi_1= \phi_0 + \Lambda \phi_0,
\end{equation*}
 is a static solution to \eqref{Eq:Evol_equ}. To study evolutions of initial data near $\Psi_0$ we consider the perturbation ansatz
\begin{equation*}
	\Psi(\tau) = \Psi_0 + \Phi(\tau).
\end{equation*} 
This leads to the central evolution equation of the paper
\begin{equation}\label{Eq:Vector_pert}
	\partial_\tau \Phi(\tau) = \big[\widetilde{\bm L}_0 + \mb L' \big] \Phi(\tau) + \bm N(\Phi(\tau)), 
\end{equation}
where
\begin{equation}\label{Def:L'}
	\mb L'\mb u = 
	\begin{pmatrix}
		0 \vspace{1mm} \\ Vu_1
	\end{pmatrix}
	\quad \text{for} \quad  V(\xi)=\frac{n-3}{2|\xi|^2}\eta'(|\xi|\phi_0)=\frac{8(n-4)(n-3)}{(|\xi|^2+n-4)^2},
\end{equation}
and
\begin{equation}\label{Def:N}
	\mb N(\mb u) = 
	\begin{pmatrix}
		0 \vspace{1mm} \\ N(\cdot,u_1)
	\end{pmatrix}
	\quad \text{for} \quad  N(\xi,u_1)=\frac{n-3}{2|\xi|^3}\Big(\eta\big(|\xi|(\phi_0+u_1)\big)- \eta(|\xi|\phi_0)-\eta'(|\xi|\phi_0)|\xi|u_1 \Big).
\end{equation}
Furthermore, the initial data are now
\begin{equation}\label{Eq:SS_initial_data}
	\Phi(0) = \Psi(0) - \Psi_{0} =
	\begin{pmatrix}
		Tv_0(T\cdot) - \phi_0 \\
		T^2 v_1(T\cdot) - \phi_1
	\end{pmatrix}
	=
	\begin{pmatrix}
		Tv_0(T\cdot) - v_0 \\
		T^2 v_1(T\cdot) - v_1
	\end{pmatrix}
	+
	\mb v := \mb U (\mb v,T),
\end{equation}
where, for convenience, we denoted
\begin{equation}\label{Def:InitCond_v}
	\mb v := 
	\begin{pmatrix}
		v_0 - \phi_0 \\
		v_1 - \phi_1
	\end{pmatrix}.
\end{equation}
To study the Cauchy problem \eqref{Eq:Vector_pert}-\eqref{Eq:SS_initial_data} we need a convenient functional setup. For this, we pass to the following section.

\section{Functional setup}\label{Sec:Funct_setup}
 
\noindent For $u,v \in C^{\infty}_{c}(\R^n)$ and $s>-n/2$ we define the inner product
\begin{equation*}
	\langle u,v \rangle_{\dot{H}^s(\R^n)}:=  \inner{|\cdot|^s \mc Fu,|\cdot|^s\mc Fv}_{L^2(\R^n)},
\end{equation*}
where we use the following convention for the Fourier transform 
\begin{equation*}
	\mc Fu(\xi) := \frac{1}{(2 \pi)^{n/2}}\int_{\R^n} u(x)e^{-i \xi \cdot x}dx.
\end{equation*}
Consequently, we have the homogeneous Sobolev norm on $C^{\infty}_{c}(\R^n)$ 
\begin{equation*}
	\norm{u}^2_{\hsob{s}}:= \inner{u,u}_{\hsob{s}}.
\end{equation*}
When working with integer values of the Sobolev exponent $s$, we will frequently use an equivalent norm defined via derivatives. Namely, we will make use of the fact that given $k \in \mathbb{N}_0$
\begin{equation}\label{Eq:Norms_equiv}
	\norm{u}_{\hsob{k}} \simeq \sum_{|\alpha|=k}\norm{\partial^\alpha u}_{L^2(\R^n)} 
\end{equation}
for all $u \in C^\infty_{c}(\R^n)$. For notational convenience, we also introduce the following inner product on $C^\infty_{c}(\R^n)$
\begin{equation*}
	\inner{u,v}_{\dot{H}^{s_1}\cap\hsob{s_2}} := \inner{u,v}_{\hsob{s_1}} + \inner{u,v}_{\hsob{s_2}},
\end{equation*}
which induces the norm $ \norm{u}_{\dot{H}^{s_1}\cap\hsob{s_2}}$.  
To introduce the central space of our analysis, we define one more inner product. For $\mb u := (u_1,u_2)$ and $\mb v := (v_1,v_2)$, both of which belong to $C^\infty_{c}(\R^n) \times C^\infty_{c}(\R^n)$, we let
\begin{equation}\label{Def:Inner_prod_H}
	\inner{\mb u,\mb v}_{\mc H^{s_1,s_2}} := \inner{u_1,v_1}_{\dot{H}^{s_1}\cap\hsob{s_2}} + \inner{u_2,v_2}_{\dot{H}^{s_1-1}\cap\hsob{s_2-1}}.
\end{equation}
Consequently, for $s_1,s_2 > -\frac{n}{2}+1$, we define the space $\mc H^{s_1,s_2}$ as the completion of the radial test space $C^{\infty}_{c,\text{rad}}(\R^n) \times C^{\infty}_{c,\text{rad}}(\R^n)$ under the norm defined by the inner product \eqref{Def:Inner_prod_H}. 
For simplicity, we do not explicitly indicate the dependence of $\mc H^{s_1,s_2}$ on the underlying spatial dimension, as we always assume it is denoted by $n$.
Finally, we remark that $\mc H^{s_1,s_2}$ is, in effect, the Cartesian product of the intersection radial Sobolev spaces
\begin{equation*}
	\mc H^{s_1,s_2}= \dot{H}_{\text{rad}}^{s_1} \cap \dot{H}_{\text{rad}}^{s_2}(\R^n) \, \times \, \dot{H}_{\text{rad}}^{s_1-1} \cap \dot{H}_{\text{rad}}^{s_2-1}(\R^n) .
\end{equation*}

\section{Existence of the free semigroup on $\mc H^{s_1,s_2}$}\label{Sec:Free_semigroup}

\noindent In order to construct in $\mc H^{s_1,s_2}$ solutions to equation \eqref{Eq:Vector_pert}, we recast it first into an integral form. This necessitates first solving the linear version, and to accomplish this we resort to semigroup theory. In this section we concentrate on the free operator $\widetilde{\mb L}_0$, and show that it generates a semigroup on $\mc H^{s_1,s_2}$. First, we supply $\widetilde{\mb L}_0$ with a domain
\begin{equation*}
	\mc D(\widetilde{\mb L}_0) := C^{\infty}_{c,\text{rad}}(\R^n) \times C^{\infty}_{c,\text{rad}}(\R^n).
\end{equation*}
This whole section is devoted to proving the following fundamental result.
\begin{proposition}\label{Prop:Free_semigroup}
	Suppose that 
	\begin{equation}\label{Eq:Conds_free_semigroup}
		 n \geq 3, \quad 1 < s_1 < \tfrac{n}{2} \quad \text{and} \quad s_2 > \tfrac{n}{2} +1.
	\end{equation}
	Then the operator $\widetilde{\bf L}_0 : \mc D(\widetilde{\bf L}_0) \subseteq \mc H^{s_1,s_2} \rightarrow \mc H^{s_1,s_2}$ is closable, and its closure $(\bm L_0, \mc D(\bm L_0))$ generates a strongly continuous semigroup $({\bf S}_0(\tau))_{\tau\geq 0}$ of bounded operators on $\mc H^{s_1,s_2}$. Furthermore, the semigroup obeys the following growth estimate
	\begin{equation}\label{Eq:S_0_decay}
		\norm{{\bf S}_0(\tau){\bf u}}_{\mc H^{s_1,s_2}} \leq e^{(\frac{n}{2}-1-s_1)\tau} \norm{{\bf u}}_{\mc H^{s_1,s_2}}
	\end{equation}
for ${\bf u} \in \mc H^{s_1,s_2}$ and $\tau \geq 0$.
\end{proposition}
 To establish this result, we use the Lumer-Phillips theorem. This requires dissipativity of (a shift of) $\bm L_0$ and that $\la \bm I - \bm L_0 $ is onto for some, large enough, $\la \in \R.$  We show these properties in a series of auxiliary results.
\begin{lemma}\label{Lem:LM_estimate}
	 Let $n \in \mathbb{N}$ and $s_2\geq s_1 > -\frac{n}{2}+1$. Then	
\begin{equation}\label{Eq:LM_est}
	\Re \, \inner{\widetilde{\bf L}_0 {\bf u},{\bf u}}_{\mc H^{s_1,s_2}} \leq \left(\frac{n}{2}-1-s_1 \right)\inner{\bm u,\bm u}_{\mc H^{s_1,s_2}}
\end{equation}
for ${\bf u} \in C^{\infty}_{c,\emph{rad}}(\R^n) \times C^\infty_{c,\emph{rad}}(\R^n)$.
\end{lemma}
\begin{proof}
	First, given $u \in \core$ we note that $\mc F(\Lambda u) = -n \mc Fu-\Lambda \mc Fu$. Accordingly, for $s>-\frac{n}{2}$ we have that
	\begin{align}
		\inner{\Lambda u,u}_{\hsob{s}} &= \inner{\mc |\cdot|^s \mc F(\Lambda u),|\cdot|^s \mc Fu }_{L^2(\R^n)} \nonumber \\
		&=-n \inner{\mc |\cdot|^s \mc Fu,|\cdot|^s \mc Fu }_{L^2(\R^n)} - \inner{ \Lambda \mc Fu,|\cdot|^{2s} \mc Fu }_{L^2(\R^n)}. \label{Eq:Lumer_Lambda}
	\end{align}
For the second term above, by partial integration we have that
\begin{equation*}
	\inner{ \Lambda \mc Fu,|\cdot|^{2s} \mc Fu }_{L^2(\R^n)} = -n \inner{ \mc Fu,|\cdot|^{2s} \mc Fu }_{L^2(\R^n)} - \inner{ \mc Fu,\Lambda(|\cdot|^{2s} \mc Fu) }_{L^2(\R^n)}, 
\end{equation*}
wherefrom we get
\begin{equation*}
	\Re \, \inner{ \Lambda \mc Fu,|\cdot|^{2s} \mc Fu }_{L^2(\R^n)} =-\left(\frac{n}{2} + s\right)\inner{\mc |\cdot|^s \mc Fu,|\cdot|^s \mc Fu }_{L^2(\R^n)}.
\end{equation*} 
Now, this, according to \eqref{Eq:Lumer_Lambda}, yields
\begin{equation*}
	\Re \, \inner{\Lambda u,u}_{\hsob{s}} = \left( s-\frac{n}{2}\right)\inner{u,u}_{\hsob{s}}.
\end{equation*}
Finally, from this identity and the fact that $\inner{\Delta u,u}_{\hsob{s}} = -\inner{u,u}_{\hsob{s+1}}$, the relation \eqref{Eq:LM_est} follows.
\end{proof}

\begin{corollary}
Let $n \in \mathbb{N}$ and $s_2\geq s_1 > -\frac{n}{2}+1$. Then the operator $\widetilde{\bf L}_0 : \mc D(\widetilde{\bf L}_0) \subseteq \mc H^{s_1,s_2} \rightarrow \mc H^{s_1,s_2}$ is closable, and its  closure $(\bm L_0, \mc D(\bm L_0))$ satisfies
\begin{equation}\label{Eq:LM_est_closed}
	\Re \, \inner{\bm L_0 {\bf u},{\bf u}}_{\mc H^{s_1,s_2}} \leq \left(\frac{n}{2}-1-s_1 \right)\inner{\bm u,\bm u}_{\mc H^{s_1,s_2}}
\end{equation}
for all ${\bm u} \in \mc D(\bm L_0)$.
\end{corollary}
\begin{proof}
	Lemma \ref{Lem:LM_estimate} implies that $\widetilde{\bf{L}}_0-(\tfrac{n}{2}-1-s_1)\mb I : \mc D(\widetilde{\mb L}_0) \subseteq \mc H^{s_1,s_2} \rightarrow \mc H^{s_1,s_2}$ is dissipative. Since this operator is also densely defined, it (and therefore also $\widetilde{\mb L}_0$) is closable (see, e.g.,~\cite{EngNag00}, p.~82, Proposition 3.14, (iv)). Then \eqref{Eq:LM_est} and the closedness of $\bm L_0$ imply \eqref{Eq:LM_est_closed}.
\end{proof}
 For notational convenience, we do not explicitly indicate the dependence of $(\bm L_0,\mc D(\bm L_0))$ on $s_1,s_2$, as that will be clear from the context. 
Now we prove an important embedding property of $\mc H^{s_1,s_2}$.

\begin{lemma}\label{Lem:L_infty_embedding}
	Let
	\begin{equation}\label{Eq:Conds_for_L_infty_embedding}
		n,m \in \mathbb{N}, \quad  0 \leq \tilde{s}_1 \leq s_1 < \tfrac{n}{2}, \quad \text{and} \quad  \tilde{s}_2 \geq s_2 > \tfrac{n}{2}+m.
	\end{equation}
	 Then we have the continuous embeddings
	\begin{equation}\label{Eq:L_infty_embedding}
	\mc H^{\tilde{s}_1,\tilde{s}_2}	\hookrightarrow \mc H^{s_1,s_2} \hookrightarrow C^{m}_{\emph{rad}}(\R^n) \times C^{m-1}_{\emph{rad}}(\R^n),
	\end{equation}
where the norm on the rightmost space is the one inherited from $	W^{m,\infty}(\R^n) \times W^{m-1,\infty}(\R^n).$
\end{lemma}
\begin{proof}
Note that for any triple $a,b,c$ with the order   $a \leq b \leq c $, we have that $|x|^b \lesssim |x|^{a}+|x|^{c}$ for all $x \in \R^n$. Consequently, the estimate
$
	\norm{u}_{\hsob{b}} \lesssim \norm{u}_{\dot{H}^a \cap \hsob{c}} 
$
 holds	for all $u \in \core$, and we thereby get that
\begin{equation*}
	\norm{\bm u}_{\mc H^{s_1,s_2}} \lesssim \norm{\bm u}_{\mc H^{\tilde{s}_1,\tilde{s}_2}}
\end{equation*}
for all $\bm u \in \core \times \core$. This inequality, in turn, implies that a Cauchy sequence in $\mc H^{\tilde{s}_1,\tilde{s}_2}$ is so in $\mc H^{s_1,s_2}$, yielding a continuous inclusion operator $\iota_1: \mc H^{\tilde{s}_1,\tilde{s}_2} \rightarrow \mc H^{s_1,s_2}$. For $\iota_1$ to be an embedding, we need to show its injectivity as well. To that end, assume $(\bm u_j)_{j \in \mathbb{N}} \subseteq \core \times \core$ is Cauchy in $\mc H^{\tilde{s}_1,\tilde{s}_2}$ and converges to zero in $\mc H^{s_1,s_2}$. Then, for $\varphi \in \core$ we have that
\begin{align*}
		|\inner{|\cdot|^{\tilde{s}_1}\mc Fu_{1,j},\varphi}_{L^2(\R^n)}| &= 
	|\inner{|\cdot|^{s_1}\mc Fu_{1,j}, |\cdot|^{\tilde{s}_1-s_1}\varphi}_{L^2(\R^n)}| \\ &\leq \norm{|\cdot|^{s_1}\mc Fu_{1,j}}_{L^2(\R^n)}\norm{|\cdot|^{\tilde{s}_1-s_1}\varphi}_{L^2(\R^n)} \rightarrow 0,
\end{align*}
as $j \rightarrow \infty$. So $(|\cdot|^{\tilde{s}_1}\mc Fu_{1,j})_{j \in \mathbb{N}}$ converges to zero weakly in $L^2(\R^n)$. Since it is also Cauchy in $L^2(\R^n)$, it is norm convergent to zero. We similarly treat the rest of the indices $\tilde{s}_2,\tilde{s}_1-1$ and $\tilde{s}_2-1$, and therefore conclude that $(\bm u_j)_{j \in \mathbb{N}} $ converges to zero in $\mc H^{\tilde{s}_1,\tilde{s}_2}$ as well, and injectivity follows.
Now we prove the other embedding in \eqref{Eq:L_infty_embedding}. Assume $s_1,s_2$ satisfy \eqref{Eq:Conds_for_L_infty_embedding}. Then for a multi-index $\alpha$ with $|\alpha|=k \leq m$ we have that
	\begin{align*}
		\norm{\partial^\alpha u}_{L^\infty(\R^n)} &\lesssim \norm{|\cdot|^k\mc Fu}_{L^1(\B^n)} + \norm{|\cdot|^k\mc Fu}_{L^1(\R^n\setminus \B^n)} \\
		& \lesssim \norm{|\cdot|^{-s_1}}_{L^2(\B^n)} \norm{|\cdot|^{s_1}\mc Ff}_{L^2(\B^n)} +
		\norm{|\cdot|^{-(s_2-k)}}_{L^2(\R^n\setminus \B^n)} \norm{|\cdot|^{s_2}\mc Ff}_{L^2(\R^n\setminus \B^n)} \\
		& \lesssim \norm{u}_{\hsob{s_1}} + \norm{u}_{\hsob{s_2}}
	\end{align*}	
	for all $u \in \core.$  From here, it follows that
	\begin{equation*}
		\norm{\bm u}_{W^{m,\infty}(\R^n)\times W^{m-1,\infty}(\R^n)} \lesssim \norm{\bm u}_{\mc H^{s_1,s_2}}
	\end{equation*}
for all $\bm u \in \core \times \core.$ Then, similarly to above, we get a continuous inclusion map $\iota_2 : \mc H^{s_1,s_2} \rightarrow C^{m}_{\text{rad}}(\R^n) \times C^{m-1}_{\text{rad}}(\R^n)$. To show injectivity of $\iota_2$, we let $(u_j)_{j  \in \mathbb{N}} \subseteq \core$ be a Cauchy sequence with respect to the norm $\norm{\cdot}_{\hsob{s}}$ for some $s \geq 0$, and furthermore assume that $\norm{u_j}_{L^\infty(\R^n)} \rightarrow 0$. Then the sequence $(|\cdot|^s\mc Fu_j)_{j \in \mathbb{N}}$ is Cauchy in $L^2(\R^n)$, and furthermore for $\varphi \in \core$ we have that
\begin{equation*}
	|\inner{|\cdot|^s\mc Fu_j,\varphi}_{L^2(\R^n)}| = 
	|\inner{u_j,\mc F^{-1}(|\cdot|^s\varphi)}_{L^2(\R^n)}| \leq \norm{u_j}_{L^\infty(\R^n)}\norm{\mc F^{-1}(|\cdot|^s \phi)}_{L^1(\R^n)} \rightarrow 0.
\end{equation*}
In other words, the sequence $(|\cdot|^s\mc Fu_j)_{j \in \mathbb{N}}$ converges to zero weakly in $L^2(\R^n)$. By the uniqueness of weak limits, we have norm convergence to zero in $L^2(\R^n)$, and consequently $\norm{u_j}_{\hsob{s}} \rightarrow 0$. From here, the desired injectivity follows.
\end{proof}

 According to Lemma \ref{Lem:L_infty_embedding}, if $s_1,s_2$ satisfy \eqref{Eq:Conds_for_L_infty_embedding}, then we interpret elements of $\mc H^{s_1,s_2}$ as functions belonging to $C^{m}_{\text{rad}}(\R^n) \times C^{m-1}_{\text{rad}}(\R^n)$. This interpretation allows us to state the following result.

\begin{lemma}\label{Lem:Decay_functions}
 Assume \eqref{Eq:Conds_free_semigroup} and	let $\bm u=(u_1,u_2) \in C^\infty_{\emph{rad}}(\R^n) \times C^\infty_{\emph{rad}}(\R^n)$ be such that for all multi-indices $\alpha \in \mathbb{N}_0^n$ for which $|\alpha| \leq s_2+1$, we have the following estimates
	\begin{equation}\label{Eq:Decay_assumpt}
		|\partial^{\alpha}u_1(x) | \lesssim \inner{x}^{s_1-\frac{n}{2}-|\alpha|-1} \quad \text{and} \quad |\partial^{\alpha}u_2(x) | \lesssim \inner{x}^{s_1-\frac{n}{2}-|\alpha|-2}
	\end{equation}
for all $x \in \R^n$. Then $\bm u \in \mc D(\bm L_0)$ and $\bm L_0 \bm u=\widetilde{\bm L}_0 \bm u$.
\end{lemma}

\begin{proof}
 Let $\varphi:\R^n \rightarrow [0,1]$ be a smooth radial function such that $\varphi(x)=1$ for $|x|\leq 1$ and $\varphi(x)=0$ for $|x|\geq2$. Assume that $\bm u=(u_1,u_2)$ satisfies the assumptions of the lemma, and define the sequence 
 $$
 (\bm u_j)_{j\in \mathbb{N}}:=(u_{1,j},u_{2,j})_{j\in \mathbb{N}} \subseteq \core \times \core, \quad
		u_{1,j}=\varphi(\cdot/j)u_1, \quad u_{2,j}=\varphi(\cdot/j)u_2.
	$$
Our aim is to show that both $(\bm u_j)_{j\in \mathbb{N}}$ and $(\widetilde{\bm L}_0 \bm u_j)_{j \in \mathbb{N}}$ are Cauchy in $\mc H^{s_1,s_2}$. For this, we exploit the equivalence \eqref{Eq:Norms_equiv} to show that the sequences are Cauchy in $\mc H^{k_1,k_2}$, for $k_1=\lfloor s_1\rfloor$ and $k_2=\lceil s_2\rceil$. Then we just use the estimate
\begin{equation}\label{Eq:Embbeding_estimate}
		\norm{\bm u}_{\mc H^{s_1,s_2}} \lesssim \norm{\bm u}_{\mc H^{k_1,k_2}}
\end{equation}
to conclude that they are Cauchy in $\mc H^{s_1,s_2}$ as well. To start, we take $\alpha,\beta \in \mathbb{N}_0^n$ such that $k_1 \leq |\alpha|+|\beta| \leq k_2$, and write 
$$
\partial^\alpha [\varphi(\cdot/j)]\partial^\beta u_1 = \frac{\inner{\cdot}^{|\alpha|}}{j^{|\alpha|}}\partial^\alpha \varphi(\cdot/j) \, \inner{\cdot}^{-|\alpha|} \partial^\beta u_1 .
$$
 For simplicity, denote $\varphi_j:=\left(\inner{\cdot}/j\right)^{|\alpha|}\partial^\alpha \varphi(\cdot/j)$. Now, note that the sequence $(\varphi_j)_{j \in \mathbb{N}}$ is bounded in $L^\infty(\R^n)$, and the support of $\varphi_j-\varphi_i$ is contained in $\R^n\setminus \B^n_{\min\{i,j\}}$. This, together with the fact that, according to \eqref{Eq:Decay_assumpt}, $ \inner{\cdot}^{-|\alpha|} \partial^\beta u_1$ belongs to $L^2(\R^n)$, implies that $\partial^\alpha [\varphi(\cdot/j)]\partial^\beta u_1$ is Cauchy in $L^2(\R^n)$. From this we conclude that $(u_{1,j})_{j \in \mathbb{N}}$ is Cauchy with respect to $\norm{\cdot}_{\hsob{k}}$ for every integer $k$ with $k_1 \leq k \leq k_2$. 
  Similarly, we show that $(u_{2,j})_{j\in \mathbb{N}}$ is Cauchy with respect to the same norm for $k_1-1 \leq k \leq k_2-1$. This, then, according to \eqref{Eq:Embbeding_estimate}, implies that $(\bm u_j)_{j \in \mathbb{N}}$ is Cauchy in $\mc H^{s_1,s_2}$.  In addition to this, from the definition of $\bm u_j$ it directly follows that $\bm u_j \rightarrow \bm u $ in $L^\infty(\R^n) \times L^\infty(\R^n).$
 
 Similarly to above, we also show that
 $(\Lambda u_{1,j})_{j \in \mathbb{N}}$ is Cauchy with respect to $\norm{\cdot}_{\hsob{k}}$ for $k_1 \leq k \leq k_2$, and that $(\Delta u_{1,j})_{j \in \mathbb{N}}$ and $(\Lambda u_{2,j})_{j \in \mathbb{N}}$ are both Cauchy with respect to the same norm for $k_1-1 \leq k \leq k_2-1$. Based on this, we conclude that $(\widetilde{\mb L}_0 \bm u_j)_{j \in \mathbb{N}}$ is Cauchy in $\mc H^{s_1,s_2}$.
Furthermore, we have that $\widetilde{\mb L}_0\bm  u_j \rightarrow \widetilde{\mb L}_0\bm u $ in $L^\infty(\R^n) \times L^\infty(\R^n).$
Finally, the lemma follows from the embedding \eqref{Eq:L_infty_embedding} and the fact that $\bm L_0$ is the closure of $\widetilde{\mb L}_0$.
\end{proof}

 The last missing piece for the invocation of the Lumer-Phillips theorem is the property of $\bm L_0$ that for one (and therefore all) large enough $\la \in \R$, the operator $\la \bm I - \bm L_0$ is onto.  

\begin{lemma}\label{Lem:Closed_range}
	Assume \eqref{Eq:Conds_free_semigroup}. Then $\rg  \big((\frac{n}{2}-1) \bm I - \bm L_0 \big) = \mc H^{s_1,s_2}$.
\end{lemma}
\begin{proof}
	Our approach consists of showing that for $\bm f \in \core \times \core$, the equation
	\begin{equation}\label{Eq:Nonhom_spectral_eq}
		\big((\tfrac{n}{2}-1 ) \bm I - \widetilde{\bm L}_0 \big) \bm u = \bm f
	\end{equation}
has a solution that satisfies the assumptions of Lemma \ref{Lem:Decay_functions}. We thereby conclude that $\core \times \core \subseteq \rg \, ((\tfrac{n}{2}-1 ) \bm I - \bm L_0)$, wherefrom the closedness of $\bm L_0$ and \eqref{Eq:LM_est_closed} imply the claim. 
To solve \eqref{Eq:Nonhom_spectral_eq}, we start by explicitly writing out the first component
\begin{equation}\label{Eq:First_component}
	u_2 = \Lambda u_1 + \tfrac{n}{2}u_1 - f_1.
\end{equation} Substituting this into the second component yields an equation for $u_1$
\begin{equation}\label{Eq:Elliptic_PDE}
	\Delta u_1 - \Lambda^2 u_1 - (n+1)\Lambda u_1 -\tfrac{n}{2}(\tfrac{n}{2}+1)u_1=g,
\end{equation}
where $g=-f_2 - \Lambda f_1 - (\tfrac{n}{2}+1)f_1$. Now, we pass to the radial profiles of $u_1$ and $g$. Namely, by assuming that $u_1=v(|\cdot|)$ and $g=h(|\cdot|)$, we obtain an ODE for $v$
\begin{equation}\label{Eq:ODE_nonhom}
	(1-r^2)v''(r)+\left(\tfrac{n-1}{r}-(n+2)r\right)v'(r)-\tfrac{n}{2}(\tfrac{n}{2}+1)v(r)=h(r).
\end{equation}
With the intention of solving this equation by the variation of constants formula, we first treat the homogeneous version
\begin{equation}\label{Eq:ODE_hom}
	(1-r^2)v''(r)+\left(\tfrac{n-1}{r}-(n+2)r\right)v'(r)-\tfrac{n}{2}(\tfrac{n}{2}+1)v(r)=0.
\end{equation}
 The sets of Frobenius indices of \eqref{Eq:ODE_hom} at $r=0$ and $r=1$ are  $\{0,2-n\}$ and $\{0,-\frac{1}{2}\}$ respectively. Therefore, there are Frobenius solutions $v_0$ and $v_1$ which are analytic at $r=0$ and $r=1$ respectively, and have non-zero values there. We claim that these two solutions form a fundamental system on $(0,1)$. To show this, we assume the contrary. As this means that $v_0$ is analytic at both $r=0$ and $r=1$, there is a power series expansion
\begin{equation*}
 	v_0(r)= \sum_{i=0}^{\infty}a_i r^{2i}, \quad a_0 \neq 0,
 \end{equation*}
with infinite radius of convergence. Consequently, the
functions $u_1:=v_0(|\cdot|)$ and $u_2:= \Lambda u_1 +\frac{n}{2} u_1$ are smooth on $\R^n$ and radial. Furthermore, since the Frobenius indices of \eqref{Eq:ODE_hom} at $r=\infty$ are $i_1=\frac{n}{2}$ and $i_2=\frac{n}{2}+1$, we see that $u_1$ and $u_2$ satisfy the estimates \eqref{Eq:Decay_assumpt}, and therefore $\bm u := (u_1,u_2)$ belongs to $\mc D(\bm L_0)$ by Lemma \ref{Lem:Decay_functions}. However, then $\bm L_0 \bm u = \widetilde{\bm L}_0 \bm u = (\frac{n}{2}-1)\bm u $, implying that $\mb u$ violates \eqref{Eq:LM_est_closed}, a contradiction.

 Now, since $\{v_0,v_1\}$ is a fundamental system on $(0,1)$, based on the form of \eqref{Eq:ODE_hom} we get that the Wronskian is $W(v_0,v_1)(r)=c (1-r^2)^{-\frac{3}{2}}r^{1-n}$, for some $c\neq 0$. Consequently, by means of the variation of constants formula, we construct a particular solution to \eqref{Eq:ODE_nonhom} on $(0,1)$
 \begin{align}
 	v(r) &=-v_0(r)\int_{r}^1 \frac{v_1(s)}{W(v_0,v_1)(s)}\frac{h(s)}{1-s^2}ds 
 	-v_1(r)\int_{0}^r \frac{v_0(s)}{W(v_0,v_1)(s)}\frac{h(s)}{1-s^2}ds \nonumber \\
 	 &=-v_0(r)\int_{r}^1 v_1(s)\sqrt{1-s} \, h_1(s)ds 
 	-v_1(r)\int_{0}^r v_0(s)\sqrt{1-s} \, h_1(s)ds,  \label{Eq:Var_con_2}
 \end{align}
where the function $h_1(s)=c^{-1}s^{n-1}\sqrt{1+s}\,h(s)$ belongs to  $C^\infty [0,\infty)$ and has bounded support. Obviously, $v \in  C^\infty(0,1)$. We claim that $v \in C^\infty(0,1].$  To show this, we first recall that the Frobenius indices of \eqref{Eq:ODE_hom} at $r=1$ are $i_1=0$ and $i_2=-\frac{1}{2}$. Therefore, as $v_1$ is the analytic Frobenius solution at $r=1$, the other one has the form $(1-r)^{-\frac{1}{2}}v_2(r)$, for some $v_2$ which is analytic at $r=1$. Consequently, there  are constants  $c_1,c_2$ such that
\begin{equation}\label{Eq:v_0_linear}
	v_0(r) = c_1 v_1(r) + c_2 \frac{v_2(r)}{\sqrt{1-r}}
\end{equation}
for $r \in (0,1)$. Now, note that the second integral  in  \eqref{Eq:Var_con_2} converges as  $r  \rightarrow  1^-$. Denote its value by $\alpha$. Then, by substituting \eqref{Eq:v_0_linear} into \eqref{Eq:Var_con_2} we get that
\begin{equation}\label{Eq:v_at_1}
	v(r)=-c_2 \frac{v_2(r)}{\sqrt{1-r}} \int_{r}^{1} v_1(s)\sqrt{1-s} \, h_1(s)ds - \alpha v_1(r) + c_2 v_1(r)  \int_{r}^{1}v_2(s)h_1(s)ds.  
\end{equation}
From here we conclude smoothness up to $r=1$. Indeed, the second and the third term in \eqref{Eq:v_at_1} are manifestly smooth at $r=1$. For the first term, we can easily see this via the substitution $s= r +(1-r)t$, by means of which this term becomes
\begin{equation*}
	-c_2v_2(r)(1-r) \int_{0}^{1} v_1\big(r+(1-r)t\big)h_1\big(r+(1-r)t\big)\sqrt{1-t}dt.
\end{equation*}
For $r>0$, one can differentiate under the integral sign, and smoothness up to $r=1$ follows.  

Now, we extend the solution \eqref{Eq:Var_con_2} beyond $r=1$. To this end, we first note that, by analytic continuation, the two functions $v_1(r)$ and $(r-1)^{-\frac{1}{2}}v_2(r)$  form a fundamental system for \eqref{Eq:ODE_hom} on $(1,\infty)$. Then, as the Wronskian for this pair is a non-zero multiple of $(r^2-1)^{-\frac{3}{2}}r^{1-n}$, by variation of constants we get that 
\begin{equation}\label{Eq:v_beyond_1}
	v(r)=-c_2 \frac{v_2(r)}{\sqrt{r-1}} \int_{r}^{1} v_1(s)\sqrt{s-1} \, h_1(s)ds - \alpha v_1(r) + c_2 v_1(r)  \int_{r}^{1}v_2(s)h_1(s)ds  
\end{equation}
is a smooth  solution  to equation \eqref{Eq:ODE_nonhom} on $(1,\infty)$. In a similar way  to above, we conclude that this solution belongs to $C^\infty[1,\infty)$. In summary, based on \eqref{Eq:v_at_1} and \eqref{Eq:v_beyond_1}, we get that
\begin{equation*}\label{Eq:v_global}
	v(r)=-c_2 \frac{v_2(r)}{\sqrt{|r-1|}} \int_{r}^{1} v_1(s)\sqrt{|s-1|} \, h_1(s)ds - \alpha v_1(r) + c_2 v_1(r)  \int_{r}^{1}v_2(s)h_1(s)ds  
\end{equation*}
is a smooth  solution to \eqref{Eq:ODE_nonhom} on  $(0,\infty)$. Therefore, $u_1:=v(|\cdot|)$ belongs to $C^\infty(\R^n\setminus\{0\})$ and solves the PDE \eqref{Eq:Elliptic_PDE} away from zero. We claim that $u_1$ is smooth at zero as well. To show this, we do the following. Based on the  asymptotic behavior of the functions appearing in \eqref{Eq:Var_con_2}, we see that $v(r)=O(1)$  and $v'(r)=O(r)$ near $r=0$. Therefore $u_1  \in  H^1(\B^n)$. Furthermore, $u_1$ is a weak solution to \eqref{Eq:Elliptic_PDE} on $\B^n$, and by  elliptic regularity we conclude that $u_1  \in C^\infty(\B^n)$. Consequently, $u_1$, and $u_2$ given by \eqref{Eq:First_component} both belong to $C^\infty_{\text{rad}}(\R^n)$, and $\bm u :=(u_1,u_2)$ solves \eqref{Eq:Nonhom_spectral_eq} globally. 

It remains to determine the  asymptotic behavior at infinity. Since $h_1$ has bounded support, from \eqref{Eq:v_beyond_1} we see  that for large values of $r$
\begin{equation*}
	v(r)=\tilde{c}_1 v_1(r) + \tilde{c}_2 \frac{v_2(r)}{\sqrt{r-1}},
\end{equation*}
for some constants $\tilde{c}_1,\tilde{c}_2$. Therefore, for large $r$, the function $v(r)$ is equal to a linear combination of the Frobenius solutions to \eqref{Eq:ODE_hom} at $r=\infty$. Since the Frobenius indices at $r=\infty$ are $i_1=\frac{n}{2}$ and $i_2=\frac{n}{2}+1$, we get that  $u_1$ and $u_2$ satisfy the decay conditions \eqref{Eq:Decay_assumpt}. Consequently, by  Lemma \eqref{Lem:Decay_functions}, $\bm u$  belongs to 
$\mc D(\bm L_0)$, and this finishes the proof.
\end{proof}
 
\begin{proof}[Proof of Proposition \ref{Prop:Free_semigroup}] Since $\bm L_0$ is densely defined, satisfies \eqref{Eq:LM_est_closed} and Lemma \ref{Lem:Closed_range} holds, the claim follows from the Lumer-Phillips theorem (see, e.g., \cite{RenRog04}, p.~407, Theorem 12.22).
\end{proof}

\section{Existence of the perturbed semigroup on $\mc H^{s,k}$}\label{Sec:Perturbed_semigroup}

\noindent In this section we propagate Proposition \ref{Prop:Free_semigroup} to the operator $\widetilde{\mb L}:=\widetilde{\mb L}_0 + \mb L'$. First, we let $\mc D(\widetilde{\mb L}):=\mc D(\widetilde{\mb L}_0)$. Then, we note that boundedness of $\mb L'$ suffices to infer that (the closure of)  $\widetilde{\mb L}$ generates a semigroup in $\mc H^{s_1,s_2}$.  In what follows, we show that the operator $\mb L'$ is, in fact, compact. This property of $\mb L'$ will be essential for the rest of the paper. To make the argument more transparent, and in anticipation of wider applications, we prove a more general compactness result. 
\begin{lemma}
		Suppose that 
	\begin{equation*}\label{Eq:Conds_param_compactness}
		 n \geq 3, \quad 1 < s < \tfrac{n}{2}, \quad k > \tfrac{n}{2} +1, \quad k \in \mathbb{N}.
	\end{equation*}
Furthermore, assume that $V \in C^\infty_{\emph{rad}}(\R^n)$, and that given  $\alpha \in \mathbb{N}_0^n$ with $|\alpha|\leq k$
\begin{equation*}
	|\partial^\alpha V(x) | \lesssim \inner{x}^{-2-|\alpha|}
\end{equation*}
for all $x \in \R^n$.
	Then the mapping
\begin{equation}\label{Eq:Potential_map}
	\begin{pmatrix}
		u_1\\u_2
	\end{pmatrix}  \mapsto
\begin{pmatrix}
	0 \\ V u_1
\end{pmatrix} 
\end{equation}
defines a compact linear operator on $\mc H^{s,k}$.
\end{lemma}

\begin{proof}
	Our approach is as follows. We first utilize the equivalence \eqref{Eq:Norms_equiv} to show that \eqref{Eq:Potential_map} maps $\mc H^{s,k}$ into $\mc H^{\lfloor s\rfloor,k}$ compactly. Then we simply compose this map with the continuous embedding $\mc H^{\lfloor s\rfloor,k} \hookrightarrow \mc H^{s,k}$. To show compactness, it is enough to prove that the set 
\begin{equation*}
	K_{\alpha}:=\{ \partial^\alpha (Vu) : u \in \core, \norm{u}_{\dot{H}^s \cap \hsob{k}} \leq 1   \} 
\end{equation*}
is totally bounded in $L^2(\R^n)$ for all multi-indices $\alpha$ of length $\lfloor s\rfloor-1$ or $k-1$. For this, we use a suitable version of the Rellich-Kondrachov theorem given in \cite{HanHar10}, Theorem 10. To apply this result, we first convince ourselves that $K_\alpha$ is bounded in $H^1(\R^n)$. To this end, let $\beta,\gamma$ be multi-indices such that $ \lfloor s\rfloor-1 \leq |\beta| + |\gamma| \leq k $. If $|\gamma|<s$, we exploit the decay of $V$ and its derivatives to obtain the following bound
\begin{equation*}
	\| \partial^\beta V  \partial^\gamma u  \|_{L^2(\R^n)} \lesssim \norm{|\cdot|^{-s+|\gamma|}\partial^\gamma u}_{L^2(\R^n)} \lesssim \norm{u}_{\hsob{s}},
\end{equation*}
where for the second estimate we used Hardy's inequality (see, e.g., \cite{MusSch13}, p.~243, Theorem 9.5). If $|\gamma| \geq s$, then we simply have
\begin{equation*}
	\| \partial^\beta V  \partial^\gamma u  \|_{L^2(\R^n)} \lesssim \norm{\partial^\gamma u}_{L^2(\R^n)} \lesssim \norm{u}_{\dot{H}^s \cap \hsob{k}},
\end{equation*}
for all $u \in \core$.
This shows that $K_\alpha$ is bounded in $H^1(\R^n)$. In the same way, we use the decay of $V$ and its derivatives to infer the existence of $\epsilon>0$ such that
\begin{equation*}
	\| \partial^\beta V  \partial^\gamma u  \|_{L^2(\R^n\setminus \B^n_R)} \lesssim \frac{1}{R^\epsilon}\norm{u}_{\dot{H}^s \cap \hsob{k}}
\end{equation*}
for all $u \in \core$ and all $R>0$. This, according to \cite{HanHar10}, Theorem 10, implies the desired total boundedness of $K_\alpha$, and we are done.
\end{proof}

 Since our potential $V$ from \eqref{Def:L'} satisfies the assumptions of the previous proposition, we readily obtain the compactness of $\mb L'$.

\begin{corollary}
	Assume that
	\begin{equation}\label{Eq:Conds_on_s,k}
	 n \geq 5, \quad 1 < s < \tfrac{n}{2}, \quad  k > \tfrac{n}{2} + 1, \quad  k \in \mathbb{N}. 
	\end{equation}
	Then the operator $\bm L' : \mc H^{s,k} \rightarrow \mc H^{s,k}$ is compact.
\end{corollary}
Note that mere boundedness of $\mb L'$ is enough to infer that the operator $\widetilde{\mb L}$ is closable, with the closure being
\begin{equation*}
	\mb L= \mb L_0 + \mb L', \quad \text{with} \quad \mc D(\mb L)=\mc D(\mb L_0).
\end{equation*}
Then, as a direct consequence of the bounded perturbation theorem (see, e.g., \cite{EngNag00}, p.~158, Theorem 1.3), we obtain the following result.
\begin{proposition}\label{Prop:S}
	Assume \eqref{Eq:Conds_on_s,k}. Then the operator $\bm L : \mc D(\bm L) \subseteq \mc H^{s,k} \rightarrow  \mc H^{s,k}$ generates a strongly continuous semigroup $(\bm S(\tau))_{\tau \geq 0}$ of bounded operators on $\mc H^{s,k}$. Furthermore, we have that
		\begin{equation}\label{Eq:Growth_S_pert}
		\norm{{\bf S}(\tau){\bf u}}_{\mc H^{s,k}} \leq e^{(\frac{n}{2}-1-s+ \| \bf L'\|)\tau} \norm{{\bm u}}_{\mc H^{s,k}}
	\end{equation}
for $\bm u \in \mc H^{s,k}$ and $\tau \geq 0$. 
\end{proposition}

\section{Spectral analysis of $\mb L$}\label{Sec:Spec_analy_L}

\noindent Although the bounded perturbation theorem guarantees existence of a semigroup, the growth estimate it yields is typically an overshoot. To get a better handle on the growth of $\mb S(\tau)$, we resort to a spectral mapping theorem that applies to this setting. Namely, the fact that $\mb L$ is a compact perturbation of $\mb L_0$ implies that any additional growth of $\mb S(\tau)$ relative to $\mb S_0(\tau)$ is induced by the point spectrum of $\mb L$. The precise statement of this fact is given in \cite{Glo22}, Theorem B.1, which is a result we will frequently refer to throughout this section. 

In the rest of the paper, we assume that $s> \frac{n}{2}-1$. This, in particular, ensures the exponential decay of the free semigroup $\mb S_0(\tau)$; see \eqref{Eq:S_0_decay}. Then, according to the discussion above, we have that potential growth of $\mb S(\tau)$ is necessarily caused by the presence of the unstable spectrum of its generator $\mb L$. As it turns out, the only unstable spectral point of $\mb L$ is a simple eigenvalue $\la=1$, which is induced by the underlying time translation symmetry. 
\begin{proposition}\label{Prop:Pert_semigroup}
	Assume that
	\begin{equation}\label{Eq:Conds_n_geq_5}
		 n \geq 5, \quad \tfrac{n}{2} - 1 < s < \tfrac{n}{2}, \quad  k > \tfrac{n}{2} +1, \quad  k \in \mathbb{N}. 
	\end{equation}
Then for the operator $\bm L: \mc D(\bm L) \subseteq \mc H^{s,k} \rightarrow \mc H^{s,k}$ the following statements hold.
\begin{itemize}[leftmargin=6mm]
	\setlength{\itemsep}{1mm}
	\item[\emph{i)}] {Spectral gap property:} 
	There exists $ \omega \in (0, s+1-\frac{n}{2})$ such that
	\begin{equation*}
		\{ \la \in \sigma(\bm L) : \Re \la \geq -\omega \} = \{ 1 \}.
	\end{equation*} 
	\item[\emph{ii)}] {Symmetry eigenvalue:} The point $\la=1$ is a simple eigenvalue, with an explicit eigenfunction
	 \begin{equation}\label{Def:g}
		\bm g:=
		\begin{pmatrix}
			g \\ \Lambda g + 2g
		\end{pmatrix},
		\quad \text{where} \quad g(\xi)=\frac{1}{|\xi|^2+n-4}.
	\end{equation}
\end{itemize}
\end{proposition}

 \begin{proof}
 First, according to \cite{Glo22}, Theorem B.1,  we have that 
\begin{equation*}
	\{ \la \in \sigma(\bm L) : \Re \la > \tfrac{n}{2}-1-s \} \subseteq \sigma_p(\bm L).
\end{equation*}   
Furthermore, by the same theorem, to show (i) it is enough to prove that 
\begin{equation}\label{Eq:S_0_empty}
	S_0:=\{ \la \in \sigma_p(\bm L) : \Re \la \geq 0, \la \neq 1 \} = \emptyset.
\end{equation}
 To show this, we assume the contrary, that there exist $\la \in S_0$ and $\bm u \in \mc D(\bm L)$ such that
$
 	\bm L\bm u = \la \bm u.
$
This means that the first component of $\bm u$ satisfies the following equation
\begin{equation}\label{Eq:Spectral_PDE}
	\Delta u_1 - \Lambda^2 u_1 - (2 \la +3)\Lambda u_1 -(\la+1)(\la + 2)u_1 + Vu_1=0
\end{equation}
 weakly on $\R^n$. Note that by Sobolev embedding we have that $\tilde{u}_1 \in C[0,\infty) \cap  C^{\lfloor \frac{n}{2} \rfloor}(0,\infty)$, where $u_1=\tilde{u}_1(|\cdot|)$. Therefore, the function $v:=(\cdot)\tilde{u}_1$ also belongs to $C[0,\infty) \cap  C^{\lfloor \frac{n}{2} \rfloor}(0,\infty)$, and, according to \eqref{Eq:Spectral_PDE}, satisfies on $(0,\infty)$ the  following ODE
\begin{equation}\label{Eq:ODE_spectral}
	(1-r^2)v''(r)+\left(\tfrac{n-3}{r}-2(\la+1)r\right)v'(r)-\la(\la+1)v(r)-\tilde{V}(r)v(r)=0,
\end{equation}
where
\begin{equation*}
	\tilde{V}(r)=\frac{(n-3)[r^4-6(n-4)r^2+(n-4)^2]}{r^2(r^2+n-4)^2}.
\end{equation*}
Now, observe that the sets of Frobenius indices of \eqref{Eq:ODE_spectral} at $r=0$ and $r=1$ are $\{1,3-n\}$  and $\{ 0,\frac{n-3}{2}-\la \}$. Then, since $v \in C[0,\infty) \cap C^{\lfloor \frac{n}{2} \rfloor}(0,\infty)$, Frobenius theory implies that $v$ must be analytic at both $r=0$ and $r=1$; in particular, $v \in C^\infty[0,1]$. However, in \cite{CosDonGlo17}, Theorem 2.2, it is proven that \eqref{Eq:ODE_spectral} does not admit such solutions for $\la \in S_0$, a contradiction. Therefore \eqref{Eq:S_0_empty} holds, and the claim (i) follows.
 
 Now we prove (ii). 
First, we note that $\bm g$ from \eqref{Def:g} satisfies the assumptions of Lemma \ref{Lem:Decay_functions}. Then, by a straightforward calculation we get that $(\widetilde{\bm L}_0 + \bm L') \bm g =\bm g$. Therefore, $1 \in \sigma_p(\bm L)$. To study the multiplicity of this eigenvalue, we define the Riesz projection 
\begin{equation*}
	\mb P:= \frac{1}{2\pi i}\int_{\gamma}\mb R_{\mb L}(\la) \, d\la,
\end{equation*}
where $\gamma$ is a positively oriented circle centered at $1$ with radius $r_\gamma < 1$. From the definition of $\bm P$ we have that $\bm P \bm g = \bm g$, and therefore $\inner{\bm g} \subseteq \rg \bm P$. To establish (ii), it remains to prove the reversed inclusion. Since $\rg \bm P$ is finite dimensional (see \cite{Glo22}, Theorem B.1) it is enough to show that there does not exist $\bm u \in \mc D(\bm L)$ such that $(\bm I - \bm L)\bm u = -\bm g$. We assume the contrary. This implies that the radial profile $\tilde{u}_1$ of the first component $u_1$ belongs to $C[0,\infty) \cap C^{\lfloor \frac{n}{2} \rfloor}(0,\infty)$ and satisfies on $(0,\infty)$ the following equation
\begin{equation}\label{Eq:ODE_la_1}
	(1-r^2)\tilde{u}_1''(r)+\left( \tfrac{n-1}{r}-6r\right)\tilde{u}_1'(r)-6\tilde{u}_1(r)+V_1(r)\tilde{u}_1(r)=G(r),
\end{equation}
where $V_1(r)=\tfrac{8(n-3)(n-4)}{(r^2+n-4)^2}$ is the radial profile of $V$, and 
\begin{equation*}
	G(r)=2rg_1'(r)+5g_1(r)= \frac{r^2+5n-20}{(r^2+n-4)^2},
\end{equation*} 
where $g_1(r)=\tfrac{1}{(r^2+n-4)^{2}}$ is the radial profile of $g$. Since $g_1$ solves the homogeneous version of \eqref{Eq:ODE_la_1}, by reduction of order we obtain another solution 
\begin{equation}\label{Def:g_tilde}
	\tilde{g}_1(r):=g_1(r)\int_{r_1}^{r}\frac{(1-x^2)^{\frac{n-7}{2}}}{x^{n-1}}\frac{dx}{g_1(x)^2},
\end{equation}
where $r_1$ is chosen arbitrarily from $(0,1).$ Note that $\tilde{g}_1(r)=\frac{A(r)}{r^{n-2}}$ for some $A$ which is analytic at $r=0$, with $A(0)\neq 0$. In particular, $\{g_1,\tilde{g}_1 \}$ is a fundamental system for (the homogeneous version of) \eqref{Eq:ODE_la_1} on $(0,1)$, and by the variation of constants formula we get that 
 	\begin{equation}\label{Eq:v_var_con}
 	\tilde{u}_1(r)= c_1 g_1(r) + c_2 \tilde{g}_1(r) + \tilde{g}_1(r) \int_{0}^{r} \frac{g_1(x)G(x)x^{n-1}}{(1-x^2)^{\frac{n-5}{2}}}dx - g_1(r) \int_{0}^{r} \frac{\tilde{g}_1(x)G(x)x^{n-1}}{(1-x^2)^{\frac{n-5}{2}}}dx
 \end{equation}
for a choice of constants $c_1,c_2$. Due to the strong singularity of $\tilde{g}_1$ at $r=0$, we conclude that $c_2=0$. Now we separately treat cases $n=5$ and $n \geq 6$.

For $n=5$ we choose $r_1=\frac{1}{2}$ in \eqref{Def:g_tilde}. Then $\tilde{g}_1(r) \simeq \ln(1-r)$ near $r=1$. Consequently, from \eqref{Eq:v_var_con} it follows that $\tilde{u}_1(r) \simeq \ln(1-r)$, unless the second integral in \eqref{Eq:v_var_con} is zero for $r=1$. But this is not the case since the integrand is strictly positive on $(0,1)$. We therefore conclude that there is no choice of $c_1,c_2$ in \eqref{Eq:v_var_con} that yields a solution that is continuous on $[0,1]$, which is in contradiction with $\tilde{u}_1$ belonging to $C[0,1]$.

If $n \geq 6$ then we take $r_1=1$ in \eqref{Def:g_tilde}. Since $g_1$ does not vanish at $r=1$, we have that 
\begin{equation*}
	\tilde{g}_1(r)=(1-r^2)^{\frac{n-5}{2}}B(r)
\end{equation*}
where $B$ is analytic near $r=1$, and $B(1)\neq 0$. As a consequence, we see that the last term in \eqref{Eq:v_var_con} is analytic at both $r=0$ and $r=1$. It remains to understand the regularity of the middle term
\begin{equation}\label{Def:I_n}
	\mc I_n(r):= \tilde{g}_1(r)\int_{0}^{r} \frac{g_1(x)G(x)x^{n-1}}{(1-x^2)^{\frac{n-5}{2}}}dx.
\end{equation}
In what follows, we show that $\mc I_n$ is not analytic at $r=1$. This implies that
\begin{equation*}
	\mc I_n(r) \simeq 
	\begin{cases}
		~(1-r)^{\frac{n-5}{2}}  &\text{if }  n \text{ is even}, \\
		~(1-r)^{\frac{n-5}{2}}\ln(1-r)  ~ &\text{if } n \text{ is odd},
	\end{cases}
\end{equation*}
near $r=1$. Then in both cases we get a contradiction with $\tilde{u}_1 \in C^{\lfloor \frac{n}{2} \rfloor}(0,1]$. For simplicity, we denote $m = \frac{n-3}{2}$. Then, we have that
\begin{align}
	\mc I_n(r)&=\tilde{g}_1(r)\int_{0}^{r}\frac{x^{2m+2}}{(1-x^2)^{m-1}}\big(  2xg_1'(x)g_1(x)+5g_1(x)^2  \big)dx \nonumber \\
	&=B(r)(1-r^2)^{m-1} \int_{0}^{r}\frac{x^{2m-2}}{(1-x^2)^{m-1}}\frac{d}{dx}\big(  x^5g_1(x)^2  \big)dx \nonumber \\
	&=B(r) \left( r^{2m+3}g_1(r)^2-2(m-1)(1-r^2)^{m-1}\int_{0}^{r}\frac{x^{2m+2}}{(1-x^2)^{m}}g_1(x)^2dx \right). \label{Eq:I_n}
\end{align} 
We now prove that the term involving the integral in \eqref{Eq:I_n} is not analytic at $r=1$. For this, we will again resort to \cite{CosDonGlo17}. First, we note that the function
\begin{equation}\label{Def:v_tilde}
	\tilde{v}(r):=\frac{(1-r^2)^{m-1}}{r^{2m+1}g_1(r)}\int_{0}^{r}\frac{x^{2m+2}}{(1-x^2)^m}g_1(x)^2\, dx
\end{equation}
solves the equation
\begin{equation}\label{Eq:SUSY}
	(1-r^2)\tilde{v}''(r)+\left(\frac{2m}{r}-4r \right)\tilde{v}'(r)-2\tilde{v}(r)+\frac{2(2m-1)(r^2-2m-1)}{r^2(r^2+2m-1)}\tilde{v}(r)=0.
\end{equation}
For a detailed derivation of \eqref{Def:v_tilde} as a solution of \eqref{Eq:SUSY}, see \cite{ChaDonGlo17}, Appendix A. From \eqref{Def:v_tilde} it is immediate that $\tilde{v} \in C^\infty[0,1)$, and since according to \cite{CosDonGlo17}, Theorem 4.1, equation \eqref{Eq:SUSY} does not admit $C^\infty[0,1]$ solutions, it follows that $\tilde{v}$ is not analytic at $r=1$. Consequently, from \eqref{Eq:I_n} it follows that $\mc I_n $ is not analytic at $r=1$, and we are done.
 \end{proof}

Now, as another direct consequence of \cite{Glo22}, Theorem B.1, we obtain the exponential decay of $\mb S(\tau)$ on the kernel of the projection $\mb P$.

\begin{proposition}[Exponential decay on the stable subspace]\label{Prop:Decay_stable}
	Assume \eqref{Eq:Conds_n_geq_5} and let $\omega$  be from Proposition \ref{Prop:Pert_semigroup}. Then there exists $C \geq 1$ such that
	\begin{equation}\label{Eq:Decay_stable}
		\norm{\bm S(\tau)(\bm I-\bm P)\bm u}_{\mc H^{s,k}} \leq C e^{-\omega\tau}\norm{(\bm I-\bm P)\bm u}_{\mc H^{s,k}}
	\end{equation}
for all $\bm u \in \mc H^{s,k}$ and $\tau \geq 0$.
\end{proposition}

\section{Estimates of the nonlinearity}\label{Sec:Est_nonlin}
\noindent With the linear theory at hand, we now turn to the analysis of the full nonlinear system \eqref{Eq:Vector_pert}. In the rest of the paper, we will only consider real-valued functions from $\mc H^{s,k}$. Also, for simplicity, we will use the same notation, $\mc H^{s,k}$, for the corresponding real subspace. In addition, from now on we will drop the subscript in $\norm{\cdot}_{\mc H^{s,k}}$, and always assume that an unspecified norm corresponds to $\mc H^{s,k}$.

 In what follows, we show that for a range of exponents $s,k$, the nonlinear operator $\mb N$ is locally Lipschitz continuous on $\mc H^{s,k}$. This will be essential later on for the construction of strong $\mc H^{s,k}$-solutions to \eqref{Eq:Vector_pert}.
 For the statement of the ensuing lemma, we also need the following notation
\begin{equation*}
	\mc B_{\delta}:=\{ \bm u \in \mc H^{s,k} : \norm{\bm u} \leq \delta \}. 
\end{equation*}

\begin{lemma}\label{Lem:Nonlin_est}
Let $\delta >0$. Also, assume that
	\begin{equation}\label{Eq:Nonlin_conds}
	n \geq 5, \quad	\tfrac{n}{2}-1 < s \leq \tfrac{n}{2}-1+\tfrac{1}{2(n-2)}, \quad  k > n, \quad  k \in \mathbb{N}.
	\end{equation}
Then the operator $\bm N$ from \eqref{Def:N} maps $\mc H^{s,k}$ into itself, and we have that
\begin{equation}\label{Eq:Nonlin_est}
	\norm{\bm N(\bm u) - \bm N(\bm v)} \lesssim (\norm{\bm u} + \norm{\bm v})\norm{\bm u - \bm v},
\end{equation}
for all $\bm u,\bm v \in \mc B_\delta$.
\end{lemma}
\begin{proof}
 We prove the estimate \eqref{Eq:Nonlin_est} for $\mb u, \mb v \in C^\infty_{c,\text{rad}}(\R^n) \times C^\infty_{c,\text{rad}}(\R^n)$ as the claim then follows by density and the $L^\infty$-embedding. The proof relies on a Schauder estimate, which, due to its generality and anticipated broad applicability, we formulate and prove separately in Appendix \ref{Sec:Schauder}; see Proposition \ref{Prop:Schauder}. To set the stage for the application of this estimate, we first establish an auxiliary identity. Namely, if $f \in C^3(\R)$ with $f''(0)=0$, then, given $a,b,c \in \R$, by repeated application of the fundamental theorem of calculus we have that
	\begin{align*}
		f(a+&c) - f(a+b) - f'(a)(c-b) =\\
		& = \int_{a+b}^{a+c} f'(t)dt - \int_{0}^{c-b}f'(a)dt 
		= (c-b)\int_{0}^{1}\left(f'\big(a+b+x(c-b)\big)-f'(a)\right)dx \\
		& = (c-b)\int_{0}^{1}\big(b+x(c-b)\big)\int_{0}^{1}f''\left(a+y\big(b+x(c-b)\big) \right)dy\, dx \\
		& =(c-b)\int_{0}^{1} \big(b+x(c-b)\big) \int_{0}^{1} \left( a+ y\big(b+x(c-b)\big) \right) 
		\cdot \\ & \hspace{6.4cm} \cdot \int_{0}^{1}  f'''\left( z \big(a+ y(b+x(c-b))\big) \right)dz \,dy\, dx. 
		\end{align*}
	According to this identity, for the second component of $\bm N(\bm u) - \bm N(\bm v)$  we have
	\begin{align}
		N(\xi,&u_1)-N(\xi,v_1)= \nonumber \\ &=\frac{n-3}{2|\xi|^3}\Big(\eta\big(|\xi|(\phi_0+u_1)\big)-\eta\big(|\xi|(\phi_0+v_1)\big) - \eta'(|\xi|\phi_0)|\xi|(u_1-v_1) \Big) \nonumber \\
		& 
		= \frac{n-3}{2}\int_{0}^{1}\int_{0}^{1}\int_{0}^{1} (u_1-v_1)\big(v_1+x(u_1-v_1)\big)
			\left( \phi_0 + y\big(v_1+x(u_1-v_1)\big) \right) \cdot \nonumber \\ &  \hspace{5.6cm}  \cdot \eta'''\left( z |\xi| \big(\phi_0+ y(v_1+x(u_1-v_1))\big) \right)dz \,dy\, dx. \label{Eq:N_u-v}
	\end{align}
	Now, since $F(x):=\eta'''(x)=8 \cos(2x)$ satisfies the assumptions of Proposition \ref{Prop:Schauder}, the estimate \eqref{Eq:Nonlin_est} follows from \eqref{Eq:N_u-v} and \eqref{Eq:Schauder_estimate}. 
\end{proof}

\section{Construction of strong solutions}\label{Sec:Strong_sol}
\noindent With the nonlinear estimate \eqref{Eq:Nonlin_est} at hand, we are in the position to construct strong solutions to \eqref{Eq:Vector_pert}. For convenience, we copy here the underlying Cauchy problem
\begin{equation}\label{Eq:Vector_pert_2}
	\begin{cases}
		~\partial_\tau \Phi(\tau) = \bm L\Phi(\tau) + \bm N(\Phi(\tau)),\\
		~\Phi(0)=\bm U(\bm v, T).
	\end{cases}	
\end{equation}
 To analyze \eqref{Eq:Vector_pert_2}, we utilize the standard techniques from dynamical systems theory. First, we use the fact that $\bm L$ generates the semigroup $\bm S(\tau)$, to recast \eqref{Eq:Vector_pert_2} into the integral form
\begin{equation}\label{Eq:Duhamel}
	\Phi(\tau)=\mb S(\tau)\mb U(\mb v,T) + \int_{0}^{\tau}\mb S(\tau-s)\mb N(\Phi(s))ds.
\end{equation} 
Then, as $\bm S(\tau)$ decays exponentially on $\ker \mb P$, we employ a  fixed point argument to show existence of global decaying solutions for small initial data. To deal with growth caused by the presence of $\rg \mb P$ in the initial data, we use a Lyapunov-Perron type argument to suppress the growing mode by appropriately choosing the blowup time. 

We now make some technical preparations. First, we introduce the Banach space
\begin{equation*}
	\mc X := \{ \Phi \in C([0,\infty),\mc H^{s,k}) : 
	\| \Phi  \|_{\mc X} := \sup_{\tau>0}e^{\omega\tau}\|\Phi(\tau)\| < \infty
	\}, 
\end{equation*}
where $\omega$ is from Proposition \ref{Prop:Pert_semigroup}. 
Then, for $\mb u \in \mc H^{s,k}$ and $\Phi \in \mc X$
we define the correction term
\begin{equation}\label{Def:Correction_term}
	\mb C(\mb u, \Phi):= \mb P \left( \mb u + \int_{0}^{\infty}e^{- s}\mb N(\Phi(s))ds \right),
\end{equation}
and a map $\bm K_{\bm u} : \mc X \rightarrow C([0,\infty),\mc H^{s,k})$ by
\begin{equation*}
	\mb K_{\mb u}( \Phi)(\tau) := \mb S(\tau)\big(\mb u - \mb C(\mb u, \Phi) \big) + \int_{0}^{\tau}\mb S(\tau-s)\mb N(\Phi(s))ds.
\end{equation*}
We also need the following definition
\begin{equation*}
	\mc X_\delta := \{ \Phi \in \mc X: \| \Phi \|_{\mc X} \leq \delta \}.
\end{equation*}
Also, for the rest of the paper, we assume that \eqref{Eq:Nonlin_conds} holds.
\begin{proposition}\label{Prop:K}
	For all  sufficiently small $\delta>0$ and all  sufficiently large $C > 0$ the following holds. If ${\bf u} \in \mc B_{\delta/C}$ then there exits a unique $\Phi = \Phi({\bf u}) \in \mc X_\delta$ for which
	\begin{equation}\label{Eq:K(u,phi)}
		\Phi = {\bf K}_{\bf u}( \Phi).
	\end{equation}
	Furthermore, the map ${\bf u} \mapsto \Phi({\bf u}): \mc B_{\delta/C} \rightarrow \mc X$ is  Lipschitz continuous. 
		\end{proposition}
\begin{proof}
	To utilize the decay of $\bm S(\tau)$ on the stable subspace, we write  $\mb K_{\mb u}$ in the following way
	\begin{equation*}
		\mb K_{\mb u}( \Phi)(\tau)= \mb S(\tau)(\mb I- \mb P)\mb u +  \int_{0}^{\tau}\mb S(\tau-s) (\mb I- \mb P)\mb N(\Phi(s))ds - \int_{\tau}^{\infty} e^{\tau- s} \mb P \mb N(\Phi(s))ds.
	\end{equation*}
	Then, from Proposition \ref{Prop:Decay_stable} and Lemma \ref{Lem:Nonlin_est} we get that if $\Phi(s) \in \mc B_\delta$ for all $s \geq 0$ then
	\begin{equation*}\label{Eq:K(u,phi)_2}
		\| \mb K_{\mb u}( \Phi)(\tau) \| \lesssim e^{-\omega \tau}\| \mb u \| + e^{-\omega \tau} \int_{0}^{\tau}e^{\omega s}\| \Phi(s) \|^2ds + e^{\tau}\int_{\tau}^{\infty} e^{-s} \|  \Phi(s)  \|^2 ds.
	\end{equation*}
	Furthermore, if $ \mb u \in \mc B_{\delta/C}$ and $\Phi \in \mc X_{\delta}$ then the above estimate implies the bound
	\begin{equation*}
		e^{\omega \tau}\| \mb K_{\mb u}( \Phi)(\tau) \| \lesssim \tfrac{\delta}{C} + \delta^2 + \delta^2 e^{-\omega\tau}.
	\end{equation*}
	Also, we similarly get that
	\begin{equation*}
		e^{\omega \tau}\| \mb K_{\mb u}( \Phi)(\tau) - \mb K_{\mb u}( \Psi)(\tau) \| \lesssim (\delta + \delta e^{-\omega \tau})\| \Phi - \Psi \|_{\mc X}
	\end{equation*}
	for all $\Phi,\Psi \in \mc X_\delta$.
	Now, the last two displayed equations imply that for all small enough $\delta$ and for all large enough $C$, given $\mb u \in \mc B_{\delta/C}$ the operator $\mb K_{\mb u}$ is contractive on $\mc X_\delta$, with the contraction constant $\frac{1}{2}$. Consequently, the existence and uniqueness of solutions to \eqref{Eq:K(u,phi)} follows from the Banach fixed point theorem. The show continuity of the map $\mb u \mapsto \Phi(\mb u) $ we utilize Proposition \ref{Prop:Decay_stable} and the contractivity of $\mb K_{\mb u}$. Namely, we have the estimate
	\begin{align*}
		\| \Phi(\mb u)(\tau) - \Phi(\mb v)(\tau) \| &=
		\| \mb K_{\mb u}( \Phi(\mb u))(\tau) - \mb K_{\mb v}( \Phi(\mb v))(\tau)\| \\
		&\leq  	\| \mb K_{\mb u}( \Phi(\mb u))(\tau) - \mb K_{\mb v}( \Phi(\mb u))(\tau)\| + \| \mb K_{\mb v}( \Phi(\mb u))(\tau) - \mb K_{\mb v} (\Phi(\mb v))(\tau)\|  \\
		&\leq \| \mb S(\tau)(\mb I-\mb P)(\mb u-\mb v) \| + \tfrac{1}{2}\| \Phi(\mb u)(\tau) - \Phi(\mb v)(\tau) \| \\ 
		&\leq  Ce^{-\omega\tau} \| \mb u - \mb v \| + \tfrac{1}{2}\| \Phi(\mb u)(\tau) - \Phi(\mb v)(\tau) \|,
	\end{align*}
	from which the Lipschitz continuity follows.
\end{proof}

In the following lemma, we summarize the relevant boundedness and continuity properties of the initial data operator $\mb U$.

\begin{lemma}\label{Lem:U(v,t)}
	For $\delta \in (0,\frac{1}{2}]$ and ${\bf v} \in \mc H^{s,k} $ the map
	\begin{equation*}
		T \mapsto {\bf U}({\bf v},T):[1-\delta,1+\delta] \rightarrow \mc H^{s,k}
	\end{equation*}
	is continuous. In addition, we have that
	\begin{equation}\label{Eq:U(v,T)_est}
		\| {\bf U}({\bf v},T) \| \lesssim  \| {\bf v} \| + |T-1|
	\end{equation}
	for all ${\bf v} \in \mc H^{s,k} $ and all $T \in [\frac{1}{2},\frac{3}{2}]$.

\end{lemma}
\begin{proof}
	Fix $\delta \in (0,\frac{1}{2}]$ and $\mb v \in \mc H^{s,k}$. Then for $T,S \in [1-\delta,1+\delta]$ we have that
	\begin{equation}\label{Eq:T-S}
		[\mb U(\mb v,T)]_1 - [\mb U(\mb v,S)]_1 = (T^2-S^2)v_0(T|\cdot|) + S^2 \big( v_0(T|\cdot|) - v_0(S |\cdot|) \big). 
	\end{equation}
	Let $\varepsilon >0$. Then there exists $\tilde{v}_0 \in C^\infty_{c,\text{rad}}(\R^n)$ for which $\| v_0(|\cdot|) - \tilde{v}_0 \|_{\dot{H}^s \cap \dot{H}^k(\R^n)} < \varepsilon$. Then, by writing
	\begin{equation*}
		v_0(T|\cdot|) - v_0(S |\cdot|) = \big ( v_0(T|\cdot|)-\tilde{v}_0(T\cdot) \big) + \big( \tilde{v}_0(T\cdot) - \tilde{v}_0(S\cdot) \big) + \big( \tilde{v}_0(S\cdot) - v_0(S|\cdot|) \big)
	\end{equation*}
	and using the fact that 
	$\lim_{S \rightarrow T} \| \tilde{v}_0(T\cdot) - \tilde{u}_0(S\cdot)  \|_{\dot{H}^s \cap \dot{H}^k(\R^n)}=0,
	$
	from \eqref{Eq:T-S} we see that 
	\begin{equation*}
		\lim_{S \rightarrow T} \|[\mb U(\mb v,T)]_1 - [\mb U(\mb v,S)]_1 \|_{\dot{H}^s \cap \dot{H}^k(\R^n)} \lesssim \varepsilon.
	\end{equation*}
	Then, continuity follows by letting $\varepsilon \rightarrow 0$. The second component of $\mb U(\mb v,T)$ is treated in the same way. For the second part of the lemma, we write  $\mb U(\mb v,T)$ in the following way
	\begin{equation}\label{Eq:U(v,T)}
		\mb U(\mb v,T)=
		\begin{pmatrix}
			Tv_1(T\cdot)\\
			T^2 v_2(T\cdot)
		\end{pmatrix}+ 
		\begin{pmatrix}
			T \phi_0(T\cdot)-\phi_0 \\
			T^2 \phi_1(T\cdot) - \phi_1
		\end{pmatrix},
	\end{equation}
	where by $v_1$ and $v_2$ we denote the components of $\bm v$. From here, the estimate \eqref{Eq:U(v,T)_est} follows.
\end{proof}

Now, we are in the position to state and prove the central result of this section.

\begin{theorem}\label{Thm:CoMain}
	There  exist $\delta,N>0$ such that the following holds. If
	\begin{equation}\label{Eq:v_smallness}
		{\bf v} \in \mc H^{s,k}, \quad \text{the components of {\bf v} are  real-valued}, \quad \text{and} \quad \| {\bf v} \| \leq \tfrac{\delta}{N^2},
	\end{equation}
	then there exist $T \in [1-\frac{\delta}{N}, 1+\frac{\delta}{N}]$ and $ \Phi \in \mc X_{\delta}$ such that \eqref{Eq:Duhamel} holds for all $\tau \geq 0.$ 
\end{theorem}
\begin{proof}
	Lemma \ref{Lem:U(v,t)} and Proposition \ref{Prop:K} imply that for all small enough $\delta$ and all large enough $N$ we have that if $\mb v$ satisfies \eqref{Eq:v_smallness} and $T \in [1-\frac{\delta}{N} , 1+\frac{\delta}{N}]$ then there is a unique $\Phi=\Phi({\mb v},T) \in \mc X_\delta $ that solves
	\begin{equation}\label{Eq:Duhamel_C}
		\Phi(\tau) =\mb S(\tau)\big(\mb U(\mb v,T) - \mb C(\mb U(\mb v,T), \Phi ) \big) + \int_{0}^{\tau}\mb S(\tau-s)\mb N\big(\Phi (s)\big)ds.
	\end{equation}
	We remark that $\Phi(\tau)$ is real-valued for all $\tau \geq 0$, as the set of real-valued functions in $\mc H^{s,k}$ is invariant under the action of both $\bm S(\tau)$ and $\bm P$. Now, to construct solutions to \eqref{Eq:Duhamel}, we prove that there is a choice of $\delta$ and $N$ such that for any $\mb v$ that satisfies \eqref{Eq:v_smallness} there is $T=T(\mb v) \in[1-\frac{\delta}{N} , 1+\frac{\delta}{N}]$ for which
	the correction term in \eqref{Eq:Duhamel_C} vanishes.
	As $\mb C$ takes values in $\rg \mb P = \langle \mb g \rangle$, it is enough to show existence of $T$ for which
	\begin{equation}\label{Eq:Fixed_pt}
		\inner{\bm C(\mb U(\mb v,T), \Phi({\mb v},T)), \mb g}_{\mc H^{s,k}}=0.
	\end{equation} 
	We therefore consider the real function $T \mapsto \inner{\bm C(\mb U(\mb v,T), \Phi({\mb v},T)), \mb g}_{\mc H^{s,k}} $ and employ the Brouwer fixed point theorem to prove that it vanishes on $[1-\frac{\delta}{N} , 1+\frac{\delta}{N}]$. The central observation to this end is that
	\begin{equation*}
		\partial_T
		\begin{pmatrix}
			T \phi_0 (T\cdot) \\
			T^2 \phi_1(T\cdot)
		\end{pmatrix} \bigg|_{T=1} 
		=  2 \sqrt{n-4} \,  \mb g,
	\end{equation*}
	 Based on this, by Taylor's formula, from \eqref{Eq:U(v,T)}  we get that 
	\begin{equation*}
		\inner{\mb P\mb U(\mb v,T) , \mb g }_{\mc H^{s,k}}= 2 \sqrt{n-4}\,\| \mb g\|^2 \,(T-1)+R_1(\mb v,T),
	\end{equation*}
	where  $R_1(\mb v,T)$ is continuous in $T$ and $R_1(\mb v,T) \lesssim \delta/N^2$.
	Furthermore, based on the definition of the correction term $\mb C$ we similarly conclude that
	\begin{equation*}
		\inner{\bm C(\mb U(\mb v,T), \Phi({\mb v},T)), \mb g}_{\mc H^{s,k}}= 2 \sqrt{n-4}\,\| \mb g\|^2 \,(T-1) +  R_2(\mb v , T),
	\end{equation*}
	where $T \mapsto R_2(\bm v,T)$ is a continuous, real-valued function on $[1-\frac{\delta}{N} , 1+\frac{\delta}{N}]$, for which  $R_2(\mb v,T) \lesssim \delta/N^2 + \delta^2.$
	Therefore, there is a choice of sufficiently large $N$ and sufficiently small $\delta$ such that $|R_2(\mb v,T)| \leq 2 \sqrt{n-4}\|\mb g \|^2 \frac{\delta}{N}$. Based on this, we get that \eqref{Eq:Fixed_pt}  is equivalent to
	\begin{equation}\label{Eq:T}
		T=F(T)
	\end{equation}
	for some function $F$ which maps the interval $[1-\frac{\delta}{N} , 1+\frac{\delta}{N}]$ continuously into itself. Consequently, by the Brouwer fixed point theorem we infer the existence of $T \in [1-\frac{\delta}{N} , 1+\frac{\delta}{N}]$ for which \eqref{Eq:T}, and therefore \eqref{Eq:Fixed_pt}, holds. The claim of the theorem follows.
\end{proof}

\section{Upgrade to classical solutions}\label{Sec:Upgrade_to_class}
\noindent In this section we show that in case the initial data $\mb v$ in \eqref{Eq:v_smallness} are smooth and rapidly decaying, the corresponding strong solution to \eqref{Eq:Duhamel} is also smooth, and satisfies \eqref{Eq:Vector_pert_2} classically. For this, we first use the abstract semigroup theory results to upgrade  strong solutions to classical ones in the semigroup sense. Then we use repeated differentiation together with Schwarz's theorem to upgrade these to smooth solutions that solve \eqref{Eq:Vector_pert_2} classically.

\begin{proposition}\label{Prop:Upgrade_to_class}
	If $\bm v$ from Theorem \ref{Thm:CoMain} is such that its components belong to the Schwartz class $\mc S(\R^n)$, then the corresponding solution $\Phi$ to \eqref{Eq:Duhamel} belongs to $C^\infty([0,\infty)\times \R^n) \times C^\infty([0,\infty)\times \R^n)$ and satisfies \eqref{Eq:Vector_pert_2} in the classical sense.
\end{proposition}
\begin{proof}
	To make a distinction, for $j \in \mathbb{N}$ we denote by $\bm S_j(\tau)$ the semigroup generated by $\bm L$ in $\mc H^{s,k+j}$. Then, we have that the following restriction property holds 
	\begin{equation}\label{Eq:Restr_j}
		\bm S(\tau)|_{\mc H^{s,k+j}} = \bm S_j(\tau).
	\end{equation}
This is, in fact, a direct consequence of a general restriction theorem that underlies this setting, and which we, for convenience and future applications, state and prove in the appendix; see Lemma \ref{Lem:Restriction}.
Now, according to \eqref{Eq:Restr_j} and the Schauder estimate \eqref{Eq:Schauder_estimate}, from \eqref{Eq:Duhamel} we have that there exists $\omega_1 \in \R$ and a continuous function $G:[0,\infty) \mapsto [0,\infty)$ such that 
\begin{align*}
	\norm{\Phi(\tau)}_{\mc H^{s,k+1}} &\lesssim 
	e^{\omega_1 \tau} \norm{\bm U(\bm v,T)}_{\mc H^{s,k+1}} + \int_{0}^{\tau}e^{\omega_1 (\tau-s)}\norm{\bm N(\Phi(s))}_{\mc H^{s,k+1}}ds\\
	& \lesssim e^{\omega_1 \tau} \norm{\bm U(\bm v,T)}_{\mc H^{s,k+1}} + \int_{0}^{\tau}e^{\omega_1 (\tau-s)}G\big(\norm{\Phi(s)}_{\mc H^{s,k}}\big)ds.
\end{align*}
Consequently, $\Phi(\tau) \in \mc H^{s,k+1}$ for all $\tau \geq 0$. Since $\bm U(\bm v,T) \in \mc H^{s,\ell}$ for all $\ell \geq k$, we proceed inductively to get that $\Phi(\tau) \in \mc H^{s,\ell}$ for every $\ell \geq k$. Then, by the embedding \eqref{Eq:L_infty_embedding} we conclude that $\Phi(\tau) \in C^\infty(\R^n)\times C^\infty(\R^n)$ for all $\tau \geq 0$. To establish regularity in $\tau$, we do the following. Since $\bm U(\bm v,T)$ belongs to $\mc D(\bm L)$ relative to $\mc H^{s,k}$,  and $\bm N$ is locally Lipschitz continuous on $\mc H^{s,k}$,  we have that $\Phi \in C^1([0,\infty),\mc H^{s,k})$, and $\Phi$ satisfies \eqref{Eq:Vector_pert_2} in the operator sense (see, e.g., \cite{CazHar98}, p.~60, Proposition 4.3.9). What is more, as $\mc H^{s,k}$ is continuously embedded in $L^\infty(\R^n) \times L^\infty(\R^n)$ the $\tau$-derivative holds pointwise. Consequently, by (a strong version of) the Schwarz theorem (see, e.g., \cite{Rud76}, p.~235, Theorem 9.41), we  conclude that mixed derivatives of all orders in $\tau$ and $\xi$  exist, and we thereby infer smoothness of both components of $(\tau,\xi) \mapsto \Phi(\tau)(\xi)$.
\end{proof}

\begin{proof}[Proof of Theorem \ref{Thm:Main}]
	First, we utilize the equivalence of Sobolev norms of corotational maps and their radial profiles to translate the smallness assumption for $(\varphi_0,\varphi_1)$ into the one for $\mb v$ from \eqref{Def:InitCond_v}. Namely, according to \cite{Glo22}, Proposition A.5, we can choose $\varepsilon>0$  small enough such that
	\begin{equation*}
		\| U_0 - \Phi \|_{\dot{H}^{s} \cap \dot{H}^{k}(\R^d)} + 	\| U_1 - \Lambda\Phi \|_{\dot{H}^{s-1} \cap \dot{H}^{k-1}(\R^d)}  < \varepsilon	
	\end{equation*}
implies
\begin{equation*}
		\| v_0 - \phi_0 \|_{\dot{H}^{s} \cap \dot{H}^{k}(\R^n)} + 	\| v_1  - \phi_1 \|_{\dot{H}^{s-1} \cap \dot{H}^{k-1}(\R^n)}  < \tfrac{\delta}{N^2},
\end{equation*}
for $\delta,N$ from Theorem \ref{Thm:CoMain}. According to the definition of $\bm v$, this means that
$\norm{\bm v}_{\mc H^{s,k}} < \tfrac{\delta}{N^2}$. Therefore, according to Theorem \ref{Thm:CoMain}, there exists $T \in [1-\frac{\delta}{N}, 1+\frac{\delta}{N}]$, and a solution $\Phi \in C([0,\infty),\mc H^{s,k})$ to \eqref{Eq:Duhamel}, for which
\begin{equation}\label{Eq:Est_Phi}
	\norm{\Phi(\tau)}_{\mc H^{s,k}} \leq \delta e^{-\omega \tau}.
\end{equation}
Now, since the components of $\bm v$ belong by assumption to $\mc S(\R^n)$, Proposition \ref{Prop:Upgrade_to_class} implies that $\Phi$ is smooth and solves \eqref{Eq:Vector_pert_2} classically. Consequently, $\Psi = \Phi + \Psi_0$ is a classical solution to \eqref{Eq:Evol_equ}. Therefore, since $\Psi(\tau) = (\psi_1(\tau,\cdot),\psi_2(\tau,\cdot))$, by defining $\tilde{\psi}_1(\tau,|\cdot|):=\psi_1(\tau,\cdot)$ we have according to \eqref{Def:psi_1}, \eqref{Def:psi}, \eqref{Def:u}, and \eqref{Def:CorAnsatz} that
\begin{equation*}
	U(t,x)=\frac{x}{T-t}\tilde{\psi}_1\left(\ln\left(\frac{T}{T-t}\right),\frac{|x|}{T-t}\right)
\end{equation*}
belongs to $C^\infty([0,T)\times \R^d)$ and solves the system \eqref{Eq:WM} on $[0,T)\times \R^d$ classically. Uniqueness of the  solution $U$ follows by standard results concerning wave equations in physical coordinates. Let $\Phi(\tau)=(\varphi_1(\tau,|\cdot|),\varphi_2(\tau,|\cdot|))$. From the fact that $\tilde{\psi}_1(\tau,|\cdot|)=\phi(|\cdot|)+\varphi_1(\tau,|\cdot|)$, we get the decomposition
\begin{equation*}
	U(t,x)=\Phi\left(\frac{x}{T-t}\right) + \varphi\left(t,\frac{x}{T-t}\right),
\end{equation*}
where $\varphi(t,x)=x\, \varphi_1\big(\ln T - \ln (T-t),|x|\big)$. We hope no confusion is caused by the small clash of notation here, since we used the letter $\Phi$ to denote both the solution \eqref{Eq:Est_Phi} and the corotational similarity profile in \eqref{Def:BB_sol_vector}.
Now, from \eqref{Eq:Est_Phi}, and the equivalence given in \cite{Glo22}, Proposition A.5, we have that
\begin{equation}\label{Eq:Est1}
	\norm{\varphi(t,\cdot)}_{\dot{H}^s \cap \dot{H}^k (\R^d)} \lesssim (T-t)^\omega,
\end{equation}
for all $t \in [0,T)$. Similarly, for the time derivative component
\begin{equation*}
	\partial_tU(t,x)=\frac{1}{T-t}\Lambda\Phi\left(\frac{x}{T-t}\right) + \partial_0\varphi\left(t,\frac{x}{T-t}\right)+ \frac{1}{T-t}\Lambda \varphi\left(t,\frac{x}{T-t}\right) ,
\end{equation*}
we get from \eqref{Eq:Est_Phi} that
\begin{equation}\label{Eq:Est2}
	\norm{(T-t)\partial_t\varphi(t,\cdot)+\Lambda \varphi(t,\cdot)}_{\dot{H}^{s-1} \cap \dot{H}^{k-1} (\R^d)} \lesssim (T-t)^\omega.
\end{equation}
The estimate \eqref{Eq:varphi_zero} then follows from \eqref{Eq:Est1} and \eqref{Eq:Est2}. To establish the last statement of the theorem, we do the following. First, according to the $L^\infty$-embedding of $\mc H^{s,k}$, from \eqref{Eq:Est_Phi} we have that $\| \varphi_1(\tau,\cdot) \|_{L^\infty[0,\infty)} \rightarrow 0$ as $\tau \rightarrow \infty$. Therefore, $\tilde{\psi}_1(\tau,\cdot) \rightarrow \phi$ uniformly on $[0,\infty)$, and consequently 
\begin{equation*}
	U(t,(T-t)\cdot)=(\cdot) \tilde{\psi}_1\left(\ln\left(\frac{T}{T-t}\right),|\cdot|\right) \rightarrow (\cdot) \phi(|\cdot|) = \Phi
\end{equation*}
uniformly on compact sets in $\R^d$ as $t \rightarrow T^-$. 
\end{proof}

\appendix

\section{A Schauder estimate for a general class of nonlinear operators}\label{Sec:Schauder}
\noindent In this section we establish a Schauder estimate, which is instrumental in proving Lemma \ref{Lem:Nonlin_est}. What is more, we expect this estimate to play an important role in proving analogous results in more general contexts; see Remarks \ref{Rem:Sch1} and \ref{Rem:Sch2} below.
\begin{proposition}\label{Prop:Schauder}
	Let $n \geq 5$. Also, let $F \in C^\infty(\R)$ be an even function, such that given $\ell \in \mathbb{N}_0$ 
	\begin{equation}\label{Eq:F_est}
		|F^{(\ell)}(x)| \lesssim 1 
	\end{equation}
for all $x \in \R$. Then, for every $s,k$ that satisfy
\begin{equation}\label{Eq:Schauder_conds}
		\frac{n}{2}-1 < s \leq \frac{n}{2}-1+\frac{1}{2(n-2)}, \quad  k > n, \quad  k \in \mathbb{N},
\end{equation}
 we have that
\begin{equation}\label{Eq:Schauder_estimate}
	\norm{u_1u_2u_3 F(|\cdot|v)}_{\dot{H}^{s-1} \cap \hsob{k}} \lesssim \prod_{i=1}^{3} \norm{u_i}_{\dot{H}^{s} \cap \hsob{k}} \big(1+\norm{v}_{\dot{H}^{s} \cap \hsob{k}}^{2k} \big)
\end{equation}
for all $u_1,u_2,u_3,v \in C^\infty_{c,\emph{rad}}(\R^n)$ where $v$ is real-valued.
\end{proposition}

\begin{proof}
	 First, we make the following observation. Since $F$ is even, by repeated application of the fundamental theorem of calculus, we infer the existence of function $G \in C^\infty[0,\infty)$ for which $F(x)=G(x^2)$. Furthermore, from \eqref{Eq:F_est} we have that given $\ell \in \mathbb{N}_0$
	\begin{equation*}
		|G^{(\ell)}(x)| \lesssim 1
	\end{equation*}
	for all $x \in [0,\infty)$.
	Now, choose $s,k$ that satisfy \eqref{Eq:Schauder_conds}. As usual, our approach is to estimate the integer order Sobolev norm $\dot{H}^{\lfloor s-1 \rfloor} \cap \dot{H}^k$, as it controls the one in \eqref{Eq:Schauder_estimate}. Since $\lfloor s-1 \rfloor=\lfloor \frac{n}{2}-2 \rfloor$, we fix $k_1:=\lfloor \frac{n}{2}-2 \rfloor$.
	 To bound the $\dot{H}^{k_1} \cap \dot{H}^k$ norm, we utilize the equivalence \eqref{Eq:Norms_equiv}, and thereby reduce the analysis to estimating $\partial^{\alpha}\big(u_1u_2u_3 G(|\cdot|^2v^2)\big)$ in $L^2$ for $|\alpha|\in \{k_1,k\}$. 
	 
	 We start with the case $|\alpha|=k_1$. Note that, by Leibnitz rule, it is enough to estimate the following expression
	 \begin{equation}\label{Eq:Schauder_Leibnitz}
	 	I(x):=x^\beta \, \partial^{\alpha_1}u_1\, \partial^{\alpha_2}u_2 \, \partial^{\alpha_3}u_3 \, G^{(\ell)}(|x|^2v^2) \prod_{i=1}^{2\ell}\partial^{\beta_i}v,
	 \end{equation}
 where 
 \begin{equation}\label{Eq:Conds_sum_beta}
 	\sum_{i=1}^{3}|\alpha_i|+ 	\sum_{i=1}^{2\ell}|\beta_i|+ 2\ell- |\beta|=k_1,
 \end{equation}
with the condition that $\ell \leq k_1$ and $|\beta| \leq 2\ell$. To this end, we define 
\begin{equation*}
	a_i:=|\alpha_i|+\frac{2\ell+s-k_1}{2\ell+2}, \quad 
	b_j:=|\beta_j|+\frac{2\ell+s-k_1}{2\ell+2},
\end{equation*}
for $i\in \{ 2,3\}$ and $j \in \{1,2,\dots,2\ell\}$. Now, according to \eqref{Eq:Conds_sum_beta} we have that 
$	-s+|\alpha_1| + a_2 + a_3 + \sum_{j=1}^{2\ell}b_j=|\beta|,
$
and therefore
\begin{equation}\label{Eq:Schauder_abs}
	|I(x)| \lesssim \frac{1}{|x|^{s-|\alpha_1|}}|\partial^{\alpha_1} u_1| \cdot |x|^{a_2}|\partial^{\alpha_2} u_2| \cdot |x|^{a_3}|\partial^{\alpha_3} u_3| \cdot \prod_{j=1}^{2\ell}|x|^{b_j}|\partial^{\beta_j} v|
\end{equation}
for all $x \in \R^n$.
To estimate the $L^2$ norm of the right side, we put the first term in $L^2$ and the rest of the terms in $L^\infty$. To bound the $L^2$ term we use Hardy's inequality, and for the rest we utilize a inequality \eqref{Eq:Gen_Strauss}, which we, for convenience, prove in a separate section. We are allowed to invoke these inequalities since $0<s-|\alpha_1|<\frac{n}{2}$ and $0 < a_i,b_j <\frac{n-1}{2}.$
Furthermore, having in mind \eqref{Eq:Schauder_conds}, we convince ourselves that
\begin{equation*}
	 s \leq \tfrac{n}{2}-a_i+|\alpha_i| < \tfrac{n}{2} \quad \text{and} \quad s \leq \tfrac{n}{2}-b_j+|\beta_j| < \tfrac{n}{2}.
\end{equation*}
 Finally, this process yields the bound 
\begin{equation}\label{Eq:Schauder_bound}
	\norm{I}_{L^2(
		\R^n)} \lesssim \prod_{i=1}^{3} \norm{u_i}_{\dot{H}^{s} \cap \hsob{k}} \, \norm{v}_{\dot{H}^{s} \cap \hsob{k}}^{2\ell}.
\end{equation}

It remains to treat the case $|\alpha|=k$.  We  consider  two sub-cases.
 First,  we assume that the highest derivative in \eqref{Eq:Schauder_Leibnitz} is of order at least $\frac{n}{2}-1$. Without loss of generality, we assume that is $\alpha_1$. Note that since $k>n$ we have that necessarily  $|\alpha_i|,|\beta_j| \leq k - \frac{n+1}{2}$ for $i\in \{ 2,3\}$ and $j \in \{1,2,\dots,2\ell\}$. Now, we split estimating the $L^2$ norm of $I$ on the unit ball and on its complement. For convenience, we define $a_1:=\min \{1,k-|\alpha_1|\}$. Then, we have that
\begin{equation*}
	|I(x)| \lesssim \frac{1}{|x|^{a_1}}|\partial^{\alpha_1} u_1| \cdot |\partial^{\alpha_2} u_2| \cdot |\partial^{\alpha_3} u_3| \cdot \prod_{j=1}^{2\ell}|\partial^{\beta_j} v|
\end{equation*}
for all $|x|\leq 1$.
To estimate the right side, we put the first term in $L^2$ and use Hardy's inequality, and the rest of the terms we put in $L^\infty$ and estimate them by
\begin{equation}\label{Eq:Schauder_L_infty_embedding}
	\norm{u}_{L^\infty(\R^n)} \lesssim \norm{u}_{\dot{H}^s \cap \hsob{\frac{n+1}{2}}},
\end{equation}
 thereby getting that
\begin{equation}\label{Eq:Schauder_bound_ball}
	\norm{I}_{L^2(
		\B^n)} \lesssim \prod_{i=1}^{3} \norm{u_i}_{\dot{H}^{s} \cap \hsob{k}} \, \norm{v}_{\dot{H}^{s} \cap \hsob{k}}^{2\ell}.
\end{equation}
Now, to treat the complement of the unit ball, we first define
\begin{equation*}
	a_i:=\tfrac{1}{2} +  \min \{ 1,|\alpha_i|\}, \quad 
	b_j:=\tfrac{1}{2}+\min \{ 1,|\beta_j|\},
\end{equation*} 
for $i\in \{ 2,3\}$ and $j \in \{1,2,\dots,2\ell\}$.  Then, since $|\beta| \leq \ell + \sum_{j=1}^{2\ell} \min \{1,|\beta_j| \}$ we have that $|\beta| \leq -a_1+a_2 + a_3 + \sum_{j=1}^{2\ell} b_j$, and we therefore conclude that
\begin{equation}\label{Eq:Schauder_complement}
	|I(x)| \lesssim \frac{1}{|x|^{a_1}}|\partial^{\alpha_1} u_1| \cdot |x|^{a_2}|\partial^{\alpha_2} u_2| \cdot |x|^{a_3}|\partial^{\alpha_3} u_3| \cdot \prod_{j=1}^{2\ell}|x|^{b_j}|\partial^{\beta_j} v|
\end{equation}
for all $|x|\geq 1$. Then, by  the usual  $L^2-L^\infty$  splitting, but this time using \eqref{Eq:Gen_Strauss} to estimate the $L^\infty$ norms, we get that
\begin{equation}\label{Eq:Schauder_bound_complement}
	\norm{I}_{L^2(
		\R^n\setminus\B^n)} \lesssim \prod_{i=1}^{3} \norm{u_i}_{\dot{H}^{s} \cap \hsob{k}} \, \norm{v}_{\dot{H}^{s} \cap \hsob{k}}^{2\ell}.
\end{equation}
Now we treat the second sub-case.  Namely, we assume that all of the derivatives in \eqref{Eq:Schauder_Leibnitz} are of order at most $\frac{n-3}{2}$. On the unit ball, we have that
\begin{equation*}
	|I(x)| \lesssim \frac{1}{|x|^{\frac{n-1}{2}-|\alpha_1|}}|\partial^{\alpha_1} u_1| \cdot |\partial^{\alpha_2} u_2| \cdot |\partial^{\alpha_3} u_3| \cdot \prod_{j=1}^{2\ell}|\partial^{\beta_j} v|.
\end{equation*}
Therefore, the $L^2-L^\infty$ splitting, together with Hardy's inequality and \eqref{Eq:Schauder_L_infty_embedding}, yield \eqref{Eq:Schauder_bound_ball}. For the complement, we do the following. First,  we define
\begin{equation*}
	a_i:=|\alpha_i|+\tfrac{3}{4}, \quad 
	b_j:=|\beta_j|+\tfrac{3}{4},
\end{equation*}
and observe that
\begin{equation*}\label{Eq:Schauder_abs_2}
	|I(x)| \lesssim \frac{1}{|x|^{\frac{n-1}{2}-|\alpha_1|}}|\partial^{\alpha_1} u_1| \cdot |x|^{a_2}|\partial^{\alpha_2} u_2| \cdot |x|^{a_3}|\partial^{\alpha_3} u_3| \cdot \prod_{j=1}^{2\ell}|x|^{b_j}|\partial^{\beta_j} v|,
\end{equation*}
for all $|x| \geq 1$. Then,  by Hardy's inequality and \eqref{Eq:Gen_Strauss}, we get  \eqref{Eq:Schauder_bound_complement}. This  finishes the proof.
\end{proof}

\begin{remark}\label{Rem:Sch1}
	We emphasize that the strength of this proposition lies in the fact that it applies to any smooth and even function $F$ with   bounded derivatives of  all orders. This, in particular, yields existence of $\mc H^{s,k}$ spaces in which one has local Lipschitz continuity for nonlinearities appearing in the corotational wave maps equation for arbitrary compact target manifolds.
\end{remark}
\begin{remark}\label{Rem:Sch2}
	 Note that \eqref{Eq:Schauder_estimate} provides more than what we need for the nonlinear estimate \eqref{Eq:Nonlin_est}. Namely, the norm we are bounding in \eqref{Eq:Schauder_estimate} controls the $\dot{H}^{s-1} \cap \dot{H}^{k-1}$ norm appearing in \eqref{Eq:Nonlin_est}. Note, however, that this strong version of the Schauder estimate is crucially used in the proof of Proposition \ref{Prop:Upgrade_to_class}. In addition to this, the anticipated benefit of this estimate is the fact that it yields control of the nonlinearity $N(\cdot,u)$ as an operator on $\dot{H}^{s} \cap \dot{H}^{k}$, which is likely to be useful in the treatment of the analogous problems in parabolic flows.
\end{remark}

\section{A Sobolev inequality for derivatives of radial functions}\label{Sec:Sobolev_ineq}

\noindent In this section we prove what appears to be a new inequality for the weighted $L^\infty$-norms of derivatives of radial Sobolev functions. The power of this estimate lies in the fact that, among other anticipated uses, it allows for short and elegant proofs of Schauder estimates like the one from Appendix \ref{Sec:Schauder}.

	\begin{proposition}\label{Prop:Strauss}
	Assume
$
	 n \geq 2$ and $\frac{1}{2} < s < \frac{n}{2}.
$
	Then, given $\alpha \in \mathbb{N}_0^n$ we have that
	\begin{equation}\label{Eq:Gen_Strauss}
		\| |\cdot|^{\frac{n}{2}-s} \partial^\alpha u\|_{L^\infty(\R^n)} \lesssim \| u \|_{\dot{H}^{|\alpha|+s}(\R^n)} 
	\end{equation}
	for all $u \in C^\infty_{c,\emph{rad}}(\R^n)$.
\end{proposition}
\begin{proof}
	
To establish this proposition, we use the $J$-Bessel function representation of the Fourier transform of radial functions. Namely, we have that for $u \in C^\infty_{c,\text{rad}}(\R^n)$
\begin{equation}\label{Eq:J_Bes}
	u(x) = \int_{0}^{\infty} \frac{J_{\frac{n}{2}-1}(|x|r)}{(|x|r)^{\frac{n}{2}-1}}\hat{u}(r)r^{n-1}dr,
\end{equation}
where $\hat{u}(|\cdot|)=\mc Fu$; see, e.g., \cite{Gra08}, p.~429. For convenience, we define
\begin{equation*}
	H(x):=\frac{J_{\frac{n}{2}-1}(|x|)}{|x|^{\frac{n}{2}-1}}.
\end{equation*}
From the properties of $J$-Bessel functions (see, e.g., ~\cite{NIST10}) we have that $H \in C_{\text{rad}}^\infty(\R^n)$ and given $\alpha \in \mathbb{N}_0^n$
\begin{equation}\label{Eq:H_est}
	|\partial^\alpha H(x)| \lesssim \inner{x}^{-\frac{n-1}{2}}
\end{equation}
for all $x \in \R^n$. Now, from \eqref{Eq:J_Bes} we get by dominated convergence theorem that
\begin{equation*}
	\partial^\alpha u(x) = \int_{0}^{\infty}\partial^\alpha H(rx) \,\hat{u}(r)r^{n+|\alpha|-1}dr.
\end{equation*}
Then, by Cauchy-Schwarz and \eqref{Eq:H_est} we get that given $\frac{1}{2} < s < \frac{n}{2}$
\begin{align*}
	|\partial^\alpha u(x)| &\leq \left( \int_{0}^{\infty}|\partial^\alpha H(rx)|^2 r^{-2s+n-1}dr \right)^{\frac{1}{2}}\left( \int_{0}^{\infty}|\hat{u}(r)|^2 r^{2s+2|\alpha|+n-1}dr \right)^{\frac{1}{2}}\\
	& \lesssim 
	 \left( \int_{0}^{\infty}\inner{rx}^{1-n} r^{-2s+n-1}dr \right)^{\frac{1}{2}}\left( \int_{0}^{\infty}|\hat{u}(r)|^2 r^{2s+2|\alpha|+n-1}dr \right)^{\frac{1}{2}} \\
	 &\lesssim |x|^{s-\frac{n}{2}} \left( \int_{0}^{\infty}(1+r^2)^{\frac{1-n}{2}} r^{-2s+n-1}dr \right)^{\frac{1}{2}} \| u \|_{\dot{H}^{s+|\alpha|}(\R^n)},
\end{align*}
for all $u \in C^\infty_{c,\text{rad}}(\R^n)$ and all $x \neq 0$.
From here, the estimate \eqref{Eq:Gen_Strauss} follows.
\end{proof}
	
\section{On restriction of semigroups}\label{Sec:Restr_semigr}

\noindent In this section we establish a restriction property of semigroups acting on two Banach spaces for which one is embedded in the other. In addition to crucially using it in Section \ref{Sec:Upgrade_to_class}, we expect this result to play an important role in regularity arguments in more general contexts.

\begin{lemma}\label{Lem:Restriction}
	Let $(X,\|\cdot\|_X)$ and $(Y,\|\cdot\|_Y)$ be separable Banach spaces. Furthermore, assume the following hold.
	\begin{itemize}[leftmargin=5mm]
		\setlength\itemsep{1mm}
		\item[\tiny$\bullet$] The space $X$ is continuously embedded in $Y$.
		\item[\tiny$\bullet$] There is a vector space $Z$ which is contained and dense in both $X$ and $Y$.
		\item[\tiny$\bullet$] There is a linear operator $ L:Z \rightarrow Z$ which is closable in both $X$ and $Y$. We furthermore denote by
		$$
		L_X:\mc D ( L_X) \subseteq X \rightarrow X
		\quad \text{and}  \quad 
		L_Y:\mc D ( L_Y) \subseteq Y \rightarrow Y
		$$
		 the closures of $L$ in $X$ and $Y$ respectively.
		 \item[\tiny$\bullet$] The operators $L_X$ and $ L_Y$ generate strongly continuous semigroups $(S_X(t))_{t \geq 0} \subseteq \mc B(X)$ and $( S_Y(t))_{t \geq 0} \subseteq \mc B(Y)$ respectively.
	\end{itemize}
\smallskip
  Then $S_X(t)$ is the restriction of $S_Y(t)$ to $X$ for all $t \geq 0$. 
\end{lemma}

\begin{proof} What follows is an adaptation of the proof of a special case of this lemma which appeared in the work of Csobo, Sch\"orkhuber and the author; see \cite{CsoGloSch21}, Lemma 3.5.   
	First, we show that $L_X$ is a restriction of $L_Y$, i.e., we prove that
	\begin{equation}\label{Eq:Equal_L}
		\mc D( L_X) \subseteq \mc D( L_Y), \quad \text{and} \quad  L_X x =  L_Yx \quad  \text{for all} \quad x \in \mc D( L_X). 
	\end{equation} 
	If $x \in Z$ then simply from the definition of $L_X$ and $L_Y$ we have that $x \in \mc D( L_Y)$ and $ L_X x =  L x =  L_Y x$. 
	We therefore now assume that $x \in \mc D( L_X)\setminus Z$. In that case, from the closedness of $ L_X$ we infer the existence of a sequence $(x_n)_{n \in \mathbb{N}} \subseteq Z$ for which 
	\begin{equation}\label{Eq:Conv}
		x_n \xrightarrow{X} x  \quad \text{and} \quad  L x_n \xrightarrow{X}  L_X x.
	\end{equation}
	Since $X \hookrightarrow Y$, it follows that the convergences above also hold in $Y$. By the closedness of $ L_Y$ it follows that 
	$x \in \mc D( L_Y)$ and $ L_Xx =  L_Yx$, and this finishes the proof of \eqref{Eq:Equal_L}.
	
	Now, note that for $\lambda \in \rho( L_X) \cap \rho( L_Y)$, from \eqref{Eq:Equal_L} the following restriction law for the resolvents holds $ R_{ L_X}(\lambda) =  R_{ L_Y}(\lambda)|_{X}$. Therefore, given $x \in X$, by Post-Widder inversion formula (see, e.g., \cite{EngNag00}, p.~223, Corollary 5.5) and the fact that $X \hookrightarrow Y$ it follows that for every $t>0$
	\begin{equation}\label{Eq:Semigroups}
		 S_X(t)x 
		=
		\lim_{n \rightarrow \infty} \left(\tfrac{n}{t} R_{ L_X}\left(\tfrac{n}{t}\right)\right)^n x
		= 
		\lim_{n \rightarrow \infty} \left(\tfrac{n}{t} R_{L_Y}\left(\tfrac{n}{t}\right)\right)^n x
		=
		 S_Y(t)x,
	\end{equation}
	where the first limit is taken in $X$ and the second one in $Y$. This finishes the proof. 
\end{proof}

\emph{Acknowledgment.} The author would like to thank Birgit Sch\"orkhuber and Sarah Kistner for useful comments on an early version of the paper.

\bibliography{refs_WM_global_higher_dim}
\bibliographystyle{plain}

\end{document}